\documentclass{amsart}

\usepackage{tikz}
\usepackage{enumerate}
\usetikzlibrary{matrix}
\usetikzlibrary{positioning}
\tikzset{>=stealth}

\newcommand{\R}{\mathbb{R}}
\newcommand{\C}{\mathbb{C}}
\newcommand{\Z}{\mathbb{Z}}
\newcommand{\Q}{\mathbb{Q}}
\newcommand{\N}{\mathbb{N}}
\newcommand{\D}{\mathbb{D}}
\newcommand{\T}{\mathbb{T}}

\newcommand{\Lrightarrow}{\hbox to1cm{\rightarrowfill}}

\theoremstyle{plain}
\newtheorem{prop}{Proposition}[section]
\newtheorem{thm}[prop]{Theorem}
\newtheorem{lemma}[prop]{Lemma}
\newtheorem{cor}[prop]{Corollary}

\theoremstyle{definition}
\newtheorem{dfn}[prop]{Definition}

\theoremstyle{remark}

\numberwithin{equation}{section}

\begin{document}

\title[Structural aspects of Generalized Harmonic Functions]{Generalized harmonic function structures}

\date{\today}

\author{Markus Klintborg}

\address{Mathematics, Faculty of Textiles, Engineering and Business, Department of Engineering, Bor\aa s University, Sweden}

\email{markus.klintborg@hb.se}

\address{Mathematics, Faculty of Science, Centre for Mathematical Sciences, 
Lund University, Sweden}

\email{markus.klintborg@math.lu.se}

\begin{abstract}
We model generalized harmonic functions on rings of differential operators and complex function spaces. The differential operators in the Weyl-algebra $A_2$ that commute with rotations are described and leads to a natural notion for such functions. We also investigate how such functions are related, and retrieve the cellular decomposition for polyharmonic functions.   
\end{abstract}



\maketitle


\section{Introduction}

Symmetry-related concerns arise naturally in the study of generalized harmonic functions on the complex domain. For example, how is one to treat an operator like $\partial \bar \partial+z\partial+ \bar z \bar \partial$, locally and away from the origin in the lack of translation invariance? Or, in heading straight for the point, under what conditions can the action of such a ``generalized harmonic operator`` on the homogeneous components $f(|z|^2)z^m$  be described solely in terms of a corresponding action on the rotational invariant part $f(|z|^2)$, so as to allow for an analysis of each part separately? Questions like these are of great importance in the study of the various forms of representation for generalized harmonic functions. 

Later investigations, initiated by A.~Olofsson and J. ~ Wittsten \cite{OW}, have produced a rich source of examples that all suggest that a property like the one just described is somehow inherent in the differential operators that commute with rotations.  In the first part of this paper, we characterize the set of differential operators in the second Weyl-algebra $A_2=A_2(\C)$ over the complex number field $\C$ that commute with rotations, hereby referred to as $\mathfrak{R}_2=\mathfrak{R}_2(\C)$. We then show that a differential operator $D\in \mathfrak{R}_2$ commutes with rotations if and only if there exists a sequence $\{T_{m}\}_{m \in \Z}$ of ordinary differential operators in the first Weyl-algebra $A_1=A_1(\C)$ such that 
\begin{equation*}
Dp(|z|^2)\xi_m(z)=\xi_m(z)T_{m}p(|z|^2),
\end{equation*} 
holds for all $m \in \Z$ and polynomials $p(x) \in \C[x]$, where $\xi_m(z)=z^m$ for $m \geq 0$ and $\xi_m(z)=\bar z^{|m|}$ for $m < 0$. As an implication of the last statement, we then conclude that the span of each of these monomial terms $\xi_m(z)$ over the ring $\C[|z|^2]$ of such polynomials is invariant under the operators in $\mathfrak{R}_2 \subset A_2$ that commute with rotations. This part will also lead to a natural notion for generalized harmonic polynomials or functions, here given as the set of those polynomials or functions that are annihilated by some element of the $\C$-subalgebra $\mathfrak{H}_2$ of $\mathfrak{R}_2$ that is generated by the elements $z\partial, \bar z\bar \partial, \partial \bar \partial \in \mathfrak{R}_2 \subset A_2$. 

In the next part of this treatment, we discuss the various forms of representation for polynomials in the ring $\C[z,\bar z]$ in terms of some of the generalized harmonic functions that have played a dominant role up to this point. In particular, we study the generalized harmonic functions that were treated by A.~ Olofsson in \cite{O14} under the confinement to polynomials in $\C[z,\bar z]$, and derive a set of bases in terms of such functions in Lemma \ref{Cdecomplemma}. This leads to the cellular decomposition for polynomials in $\C[z,\bar z]$, as portrayed in Corollary \ref{cellulardecompthmpolynomials}, which is given in the form of   
\begin{align}
\label{polyharmonicrepintro}
p(z,\bar{z})&=w_0(z,\bar z)+(1-|z|^2)w_1(z,\bar z)+\ldots+(1-|z|^2)^{n-1}w_{n-1}(z,\bar z),
\end{align}
for some polynomials $w_j(z,\bar z) \in \C[z,\bar z]$ in fulfilment of $L_{n-1-j,n-1-j}w_j(z,\bar z)=0$, where $L_{\alpha,\alpha}$ is the differential operator 
\begin{equation*}
(1-|z|^2)\partial \bar \partial+\alpha z \partial+\alpha \bar z \bar \partial-\alpha^2 \in \mathfrak{H}_2(\C).
\end{equation*}
 
In the latter part of this treatment, we move on to complex function spaces and in particular that of smooth functions on the unit disc $\D$. We employ the results that were developed in the restriction to polynomials in $\C[z,\bar z]$, and show how the generalized harmonic functions that were treated in \cite{OK} can be represented as sums of polyharmonic functions, whose order depends on the particular integer-parametrization of such generalized harmonic functions. We then turn the question around and show that any polyharmonic function of order $n$ on $\D$ has a similar representation to that in \eqref{polyharmonicrepintro} and may be cellularly decomposed in the sense of A.~ Borichev and H.~ Hedenmalm in \cite{BH}, who were the first to provide it and gave it its name. These results are found in Proposition \ref{pqtopolpropinfinite} and Theorem \ref{cellulardecomposition}, respectively.

In summary then, our aim is to extract and identify the larger class of operators in $A_2$ that display similar behaviours to those that have appeared in connection with the generalized harmonic functions that have been treated until now, and to give a rather self-contained structural account of how some of the more exemplary such functions are related.  

Historical accounts and a more comprehensive list of references in relation to generalized harmonic functions may also be found in the papers listed throughout this text. The cellular decomposition for polyharmonic functions has also been studied in a wider context and was extended to the higher dimensional setting in \cite{Perälä}. 

\section{Rotation symmetries and the Weyl-algebra $A_2(K)$ }

A polynomial is an element in the ring $K[z,\bar{z}]$ with coefficients in a field $K$ of characteristic zero. We shall refer to $\bar{z}$ as the conjugate of the indeterminant $z$, and write $|z|^2=z\bar z$ for the product $z \cdot \bar{z}$ in $K[z,\bar{z}]$. The preliminary convention will be that $\bar{f}(z,\bar{z})$ is the polynomial obtained from $f(z,\bar{z}) \in K[z,\bar{z}]$ by interchanging its arguments, i.e. $\bar{f}(z,\bar z)=f(\bar z, z)$. We will also write $z,\bar{z}$ for the linear operators in the algebra of endo\-morphisms over $K$ that act via multiplication on polynomials in $K[z,\bar{z}]$ and define $\partial$ and $\bar{\partial}$ in the usual way as $\partial(f)=\partial f / \partial z$ and $\bar \partial(f)=\partial f / \partial \bar{z}$ for polynomials $f(z,\bar{z}) \in K[z,\bar{z}]$. These operators satisfy the usual commutator relationships and the $K$-algebra they generate is the Weyl-algebra $A_2=A_2(K)$, as described in the first chapter of J.-E Bj\"ork \cite{Bjork}.

It is well familiar that the products of form $z^{\alpha_1}\bar{z}^{\beta_1}\partial^{\alpha_2}\bar{\partial}^{\beta_2}$ for $\alpha_1,\beta_1,\alpha_2,\beta_2 \in \N$ make a basis for the Weyl-algebra $A_2$ as a vector space over $K$. It is sometimes referred to as the canonical basis for $A_2(K)$, and an element written as a linear combination of such terms is said to be in canonical form. We will denote this basis by $\mathcal{B}$, and also refer to an element $B=z^{\alpha_1}\bar{z}^{\beta_1}\partial^{\alpha_2}\bar{\partial}^{\beta_2}$ of this basis as that element parametrized by the integers $\alpha_1,\beta_1,\alpha_2,\beta_2 \in \N$. In a preliminary sense, we will take $\bar f(z,\bar z,\partial, \bar \partial)$ to be the operator $f(\bar z, z, \bar \partial, \partial)$, and refer to it as the conjugate of $f(z,\bar z,\partial, \bar \partial) \in A_2$. In this section, it will also be assumed that $K$ is the field of complex numbers $\C$.   

For a given $\gamma \in \C$, we define the endomorphism $M_{\gamma}$ on $K[z,\bar z]$ that act on a polynomial $p(z,\bar z)$ in $K[z,\bar z]$ by multiplying each of its arguments by the constant $\gamma \in \C$. In symbols, 
\begin{equation*}
M_{\gamma}p(z,\bar z)=p(\gamma z, \bar \gamma \bar z).
\end{equation*} 
We will also designate the symbol $R_{\theta}$ for the elements $M_{e^{i \theta}}$ of the circle group $\T$ that are parametrized by $e^{i \theta} \in \T$, and refer to the collection of such operators as the family of rotation operators. An element $D$ in $A_2$ is then said to commute with rotations if
\begin{equation*}
DR_{\theta}p(z,\bar z) = R_{\theta}Dp(z,\bar z), 
\end{equation*}
holds for all $\theta \in [0,2\pi)$ and every polynomial $p(z,\bar z)$ in $K[z,\bar z]$. Note that the set of all elements in $A_2$ that commute with rotations is closed under the ring operations of $A_2$, so products and sums of elements in $A_2$ that commute with rotations again commute with rotations. We shall refer to the $K$-subalgebra of $A_2$ that consists of all elements that commute with rotations by use of the symbol $\mathfrak{R}_2$.  
 
Arguably, the most familiar elements of the $K$-subalgebra $\mathfrak{R}_2$ of $A_2$ are $\partial \bar \partial$ and its powers, that together define a class of functions under the umbrella term polyharmonic functions. The operator $\partial \bar \partial$ is just the ``complex version`` of the Laplace operator, and raising this operator to the second power defines the function class commonly referred to as the biharmonic functions. It is a well known fact that polynomial expressions in $\partial \bar \partial$ characterize the operators in $A_2$ that are invariant with respect to both rotations and translations. If we drop the last of these two con\-ditions however, then plenty of examples in $A_2$ arise aside from those already mentioned. Examples include the second degree terms $z\partial$ and $\bar z \bar \partial$ and elements of the polynomial ring $K[|z|^2] \subset A_2(K)$. We proceed to show that the elements $z\bar z,z\partial, \bar z \bar \partial$ and $\partial \bar \partial$ is a generating set for the $K$-subalgebra $\mathfrak{R}_2$ of $A_2(K)$, and give both a necessary and sufficient condition under which elements in the Weyl-algebra $A_2$ commute with rotations. This will later lead us to our notion of a generalized harmonic polynomial, or more generally, a generalized harmonic function. This procedure will also produce many more examples of operators in $A_2$ that commute with rotations.

\begin{lemma}
\label{parametercondcanbasisrot}
Let $B \in \mathcal{B}$ be an element in the canonical basis for $A_2(K)$ that is parametrized by the integers $\alpha_1,\beta_1,\alpha_2,\beta_2 \in \N$. Then a necessary and sufficient condition for $B$ to commute with rotations is that 
\begin{equation}
\label{commutecond}
\alpha_1-\alpha_2=\beta_1-\beta_2.
\end{equation}
\end{lemma}
\begin{proof}
Let $\theta \in [0,2\pi)$ be a fixed number and let $B=z^{\alpha_1}\bar z^{\beta_1} \partial^{\alpha_2}\bar \partial^{\beta_2}$ be parametrized by the integers $\alpha_1,\beta_1,\alpha_2,\beta_2 \in \N$.
Write $\partial$ and $\bar \partial$ in the form $\partial_z$ and $\partial_{\bar z}$, respectively. Let $p(z,\bar z)$ be any polynomial in $K[z,\bar z]$. Then
\begin{align}
\label{rightrotation}
BR_{\theta}p(z,\bar z) &= z^{\alpha_1}\bar z^{\beta_1} \partial^{\alpha_2}\bar \partial^{\beta_2}p(e^{i\theta}z,e^{-i\theta}\bar z) \\ \notag &=e^{i(\alpha_2-\beta_2)\theta}z^{\alpha_1}\bar z^{\beta_1} \partial_w^{\alpha_2}\partial_{\bar w}^{\beta_2}p(w,\bar w)\big|_{w=ze^{i\theta}}\\ \notag &= e^{i(\alpha_2-\beta_2)\theta}z^{\alpha_1}\bar z^{\beta_1} R_{\theta} \partial^{\alpha_2} \bar \partial^{\beta_2}p(z,\bar{z}). 
\end{align}
On the other hand,
\begin{align}
\label{leftrotation}
R_{\theta}Bp(z,\bar z) &= R_{\theta}z^{\alpha_1}\bar z^{\beta_1} \partial^{\alpha_2}\bar \partial^{\beta_2}p(z,\bar z) \\ \notag &=e^{i(\alpha_1-\beta_1)\theta}z^{\alpha_1}\bar z^{\beta_1}R_{\theta}\partial^{\alpha_2}\bar \partial^{\beta_2}p(z,\bar z). 
\end{align}
Since $B \in \mathcal{B}$ commutes with rotations, 
\begin{equation*}
e^{i(\alpha_1-\beta_1)\theta}=e^{i(\alpha_2-\beta_2)\theta},
\end{equation*}
by which we conclude. 
\end{proof}
The last condition gives a necessary and sufficient condition under which the canonical basis elements in $\mathcal{B}$ commute with rotations. We would also like to show that the collection of elements in $\mathcal{B}$ that satisfy this condition is a basis for $\mathfrak{R}_2$, that is, a basis for the collection of all elements in $A_2$ that commute with rotations. We have chosen to do so in the following way, for which purpose we recall the lexicographic order defined on $\N \times \N$, where $(m,n) < (m',n')$ if either $m < m'$ or $m=m'$ and $n < n'$.  

\begin{prop}
\label{spanR2}
Let $D$ be an element in the $K$-subalgebra $\mathfrak{R}_2$ of $A_2(K)$. Then $D$ lies in the linear span of all elements in the canonical basis $\mathcal{B}$ whose parameters satisfy \eqref{commutecond}. 
\end{prop}
\begin{proof}
The operator $D \in A_2$ has a canonical representation in the form of
\begin{equation}
\label{canonicalexpansionbasisR2prop}
D=c_1z^{\alpha_{1,1}} \bar z^{\beta_{1,1}}\partial^{\alpha_{2,1}} \bar \partial^{\beta_{2,1}}+\ldots+c_nz^{\alpha_{1,n}} \bar z^{\beta_{1,n}}\partial^{\alpha_{2,n}} \bar \partial^{\beta_{2,n}},
\end{equation}
for some non-zero $c_i \in K$. Since $D$ commutes with rotations, we see from equation \eqref{rightrotation} and \eqref{leftrotation} that the sum
\begin{align*}
\label{Toperatorlinind}
T_{\theta}=\sum_{j=1}^{n}c_j \big(e^{i(\alpha_{2,j}-\beta_{2,j})\theta}-e^{i(\alpha_{1,j}-\beta_{1,j})\theta}\big)z^{\alpha_{1,j}} \bar z^{\beta_{1,j}}R_{\theta}\partial^{\alpha_{2,j}} \bar \partial^{\beta_{2,j}} \in \text{End}(K[z,\bar z]),
\end{align*}
must vanish identically on $K[z,\bar z]$ for every $\theta \in [0,2\pi)$. Introduce the lexicographic ordering on the set $\N \times \N$ and assume at no cost that the pairs $(\alpha_{2,i},\beta_{2,i})$ have been indexed in increasing order for $1 \leq i \leq n$. Set
\begin{equation*}
\label{polynomialpropbasisR2}
p_{i}(z,\bar z)=z^{\alpha_{2,i}}\bar z^{\beta_{2,i}} \in K[z,\bar z],
\end{equation*}
for $1 \leq i \leq n$. Notice that 
\begin{equation*}
\partial^{\alpha_{2,i}}\bar \partial^{\beta_{2,i}}p_j(z,\bar z)=\delta_{\alpha_{2,i},\alpha_{2,j}}\delta_{\beta_{2,i},\beta_{2,j}}\alpha_{2,j}!\beta_{2,j}!, \quad i \geq j.
\end{equation*}
In other words, unless $(\alpha_{2,i},\beta_{2,i})=(\alpha_{2,j},\beta_{2,j})$ for $i \geq j$, then the right hand side of this last equation must vanish. Now, since $T_{\theta}$ must vanish identically on $K[z,\bar z]$, it must vanish at $p_1(z,\bar z)$ in particular. Hence, 
\begin{equation*}
\label{Toperatorequation}
0=\alpha_{2,1}!\beta_{2,1}!\sum_{j=1}^{n}c_j\delta_{\alpha_{2,j},\alpha_{2,1}}\delta_{\beta_{2,j},\beta_{2,1}}\big(e^{i(\alpha_{2,j}-\beta_{2,j})\theta}-e^{i(\alpha_{1,j}-\beta_{1,j})\theta}\big)z^{\alpha_{1,j}} \bar z^{\beta_{1,j}},
\end{equation*}
in $K[z,\bar z]$ for all $\theta \in [0,2\pi)$. We have the implication 
\begin{equation*}
(\alpha_{2,i},\beta_{2,i})=(\alpha_{2,j},\beta_{2,j}) \Longrightarrow (\alpha_{1,i},\beta_{1,i}) \neq (\alpha_{1,j},\beta_{1,j}), \quad i \neq j, \quad 1 \leq i,j \leq n,
\end{equation*}
for the simple reason that the canonical basis elements in \eqref{canonicalexpansionbasisR2prop} were assumed to be distinct, or parametrized by different sets of integers $\alpha_{1,i},\beta_{1,i},\alpha_{2,i},\beta_{2,i} \in \N$ for $1 \leq i \leq n$. Since the elements $z^i\bar z^j$ are linearly independent over $K$ and form a basis for $K[z,\bar z]$, we may then conclude that
\begin{equation*}
\delta_{\alpha_{2,j},\alpha_{2,1}}\delta_{\beta_{2,j},\beta_{2,1}}c_j\big(e^{i(\alpha_{2,j}-\beta_{2,j})\theta}-e^{i(\alpha_{1,j}-\beta_{1,j})\theta}\big)=0,
\end{equation*}
for $1 \leq j \leq n$. Thus, 
\begin{equation*}
\label{equalityexponentialsR2prop}
e^{i(\alpha_{2,j}-\beta_{2,j})\theta}-e^{i(\alpha_{1,j}-\beta_{1,j})\theta}=0, 
\end{equation*}
for $j \geq 1$ less than or equal to some integer $k-1 \leq n$, say, in which case \eqref{commutecond} holds. If $k-1=n$, we are done. If not, we can proceed as follows. 

Suppose that \eqref{commutecond} holds for $j \in \N$ such that $1 \leq j < k \leq n$. Similarly then, and by evaluating the polynomial $p_k(z,\bar z)$ with respect to the operator $T_{\theta}$, we get that 
\begin{equation*}
\label{Toperatorequation}
0=\alpha_{2,k}!\beta_{2,k}!\sum_{j=k}^{n}c_j\delta_{\alpha_{2,j},\alpha_{2,k}}\delta_{\beta_{2,j},\beta_{2,k}}\big(e^{i(\alpha_{2,j}-\beta_{2,j})\theta}-e^{i(\alpha_{1,j}-\beta_{1,j})\theta}\big)z^{\alpha_{1,j}} \bar z^{\beta_{1,j}},
\end{equation*}
in $K[z,\bar z]$. By repeating the argument towards the end of the last paragraph, we see in particular that 
\begin{equation*}
\label{equalityexponentialsR2prop}
e^{i(\alpha_{2,k}-\beta_{2,k})\theta}-e^{i(\alpha_{1,k}-\beta_{1,k})\theta}=0.
\end{equation*}
Inductively then, each of the canonical basis elements in \eqref{canonicalexpansionbasisR2prop} must satisfy \eqref{commutecond}. Thus, every element in $\mathfrak{R}_2$ is expressible as a sum of canonical basis elements, all of which satisfy \eqref{commutecond}.  
\end{proof}
\begin{cor}
\label{basisR2}
The set of all elements in the canonical basis $\mathcal{B}$ for $A_2(K)$ that satisfy \eqref{commutecond} is a basis for $\mathfrak{R}_2$ over $K$.
\end{cor}
\begin{proof}
The statement follows from Lemma \ref{parametercondcanbasisrot} and Proposition \ref{spanR2}.
\end{proof}
\begin{prop}
\label{basisrot}
Let $B \in \mathcal{B}$ be an element of the canonical basis for $A_2(K)$ that is parametrized by the integers $\alpha_1,\beta_1,\alpha_2,\beta_2 \in \N$. If $B$ commutes with rotations, then $B$ is a product of terms of the form $z^{\gamma_1}\bar z^{\gamma_1}, z^{\gamma_2}\partial^{\gamma_2}, \bar z^{\gamma_3} \bar \partial^{\gamma_3}$ and $\partial^{\gamma_4}\bar \partial^{\gamma_4}$ for some $\gamma_i \in \N$, factored in the order in which these elements were written.\footnote{The terms can of course be reordered to the extent that they commute.}
\end{prop}
\begin{proof}
By Lemma \ref{parametercondcanbasisrot} and equation \eqref{commutecond}, 
\begin{equation*}
\label{commutecondprop}
\beta_2=\alpha_2+\beta_1-\alpha_1 \geq 0.
\end{equation*}
Thus, we can write 
\begin{equation*}
\label{basiselementsthreeparreduction}
B=z^{\alpha_1}\bar z^{\beta_1} \partial^{\alpha_2}\bar \partial^{\beta_2}=z^{\alpha_1}\bar z^{\beta_1} \partial^{\alpha_2}\bar \partial^{\alpha_2+\beta_1-\alpha_1}.
\end{equation*}
There are six cases to consider, determined by the relative order of the integers $\alpha_1, \beta_1, \alpha_2 \in \N$.
\begin{enumerate}
\item $\alpha_1 \leq \alpha_2$.
In this case, we may write
\begin{equation*}
B=z^{\alpha_1}\bar z^{\beta_1} \partial^{\alpha_2-\alpha_1+\alpha_1}\bar \partial^{\alpha_2+\beta_1-\alpha_1}=z^{\alpha_1}\partial^{\alpha_1}\bar{z}^{\beta_1}\bar{\partial}^{\beta_1}\partial^{\alpha_2-\alpha_1}\bar \partial^{\alpha_2-\alpha_1}.
\end{equation*}
\item $\alpha_1 \leq \beta_1 $. In this case, we may write
\begin{equation*}
B=z^{\alpha_1}\bar z^{\beta_1+\alpha_1-\alpha_1} \partial^{\alpha_2}\bar \partial^{\alpha_2+\beta_1-\alpha_1}=z^{\alpha_1}\bar z^{\alpha_1}\bar z^{\beta_1-\alpha_1}\bar \partial^{\beta_1-\alpha_1}\partial^{\alpha_2}\bar \partial^{\alpha_2}.
\end{equation*} 
\item $\alpha_1 > \alpha_2$ and $\alpha_1 > \beta_1$. In this case, we may write
\begin{align*}
B &=z^{\alpha_1+\beta_1-\beta_1}\bar z^{\beta_1} \partial^{\alpha_2+\beta_1-\beta_1+\alpha_1-\alpha_1}\bar \partial^{\alpha_2+\beta_1-\alpha_1} \\ &=z^{\beta_1}\bar z^{\beta_1}z^{\alpha_1-\beta_1}\partial^{\alpha_1-\beta_1}\partial^{\alpha_2+\beta_1-\alpha_1}\bar \partial^{\alpha_2+\beta_1-\alpha_1}.
\end{align*}  
\end{enumerate}
We have covered all the six cases. 
\end{proof}

\begin{cor}
\label{basisrotgeneralsums}
Any element in the $K$-subalgebra $\mathfrak{R}_2$ of $A_2(K)$ is a sum of terms all of which can be written as a product in $z^{\gamma_1}\bar z^{\gamma_1}, z^{\gamma_2}\partial^{\gamma_2}, \bar z^{\gamma_3} \bar \partial^{\gamma_3}$ and $\partial^{\gamma_4}\bar \partial^{\gamma_4}$ for some numbers $\gamma_i \in \N$, factored in the order in which these elements were listed.
\end{cor}
\begin{proof}
The statement follows from Corollary \ref{basisR2} and Proposition \ref{basisrot}.
\end{proof}

In the conclusion that follows, we will make use of the well-known formula 
\begin{equation*}
z^n\partial^n=\sum_{m=0}^{n} s(n,m) (z\partial)^m.
\end{equation*}
It expresses the differential operators of form $z^n \partial^n$ and $\bar z^n \bar \partial^n$ as polynomials in $z \partial$ and $\bar z \bar \partial$, respectively. The coefficients of the corresponding polynomials are the Stirling numbers $s(n,m)$ of the first kind, implicitly given by
\begin{equation*}
x^{\underline{n}}=x(x-1)\ldots(x-n+1)=\sum_{m=0}^{n}s(n,m)x^m,
\end{equation*}
where $x^{\underline{n}}$ is the falling factorial.

\begin{thm}
The elements $z\bar z, z \partial, \bar z \bar \partial$ and $\partial \bar \partial$ in $A_2(K)$ is a generating set for the $K$-subalgebra $\mathfrak{R}_2$ of $A_2$ that consists of all elements in $A_2$ that commute with rotations.
\end{thm}
\begin{proof}
By Corollary \ref{basisrotgeneralsums}, each element in $\mathfrak{R}_2$ is a linear combination of terms that are products in $z^{\gamma_1}\bar z^{\gamma_1}, z^{\gamma_2}\partial^{\gamma_2}, \bar z^{\gamma_3} \bar \partial^{\gamma_3}$ and $\partial^{\gamma_4}\bar \partial^{\gamma_4}$ for some integers $\gamma_i \in \N$, taken in the order in which these elements were written. It is clear that the first and the last of these four expressions are in the algebra generated by $z\bar z, z\partial,\bar z \bar \partial$ and $\partial \bar \partial$, and by the formula preceding this statement, so are the middle two.     
\end{proof}

We shall denote the $K$-subalgebra of $A_2(K)$ that is generated by the elements $z \partial, \bar z \bar \partial$ and $\partial \bar \partial$ by the symbol $\mathfrak{H}_2=\mathfrak{H}_2(K)$. It is a subalgebra that is strictly contained in the algebra $\mathfrak{R}_2$ of all elements in $A_2(K)$ that commute with rotations. The generators of this algebra satisfy the commutator relations $[z \partial, \bar z \bar \partial]=0$ and $[\partial \bar \partial, z \partial] \ = [\partial \bar \partial, \bar z \bar \partial]= \partial \bar \partial$. We can also then add that the four dimensional space spanned by $1,z\partial, \bar z \bar \partial$ and $\partial \bar \partial $ is closed under the bracket. More specifically, let 
\begin{equation*} 
D_1=a_{1}+a_{2}z\partial+a_{3}\bar z \bar \partial + a_{4}\partial \bar \partial, \quad D_2=b_{1}+b_{2}z\partial+b_{3}\bar z \bar \partial + b_{4}\partial \bar \partial,
\end{equation*}
for some $a_{i},b_i \in K$. Then
\begin{equation*}
[D_1,D_2]=\gamma \partial\bar \partial,
\end{equation*}
where 
\begin{equation*}
\gamma=a_{4}(b_{2}+b_{3})-b_{4}(a_{2}+a_{3}).
\end{equation*}
This makes the set of such elements into a Lie subalgebra of $\mathfrak{R}_2 \subset A_2$. 

Before we proceed with some preliminary structural accounts of $K[z,\bar z]$, it can also be worth to mention that the two rings $K[z,\bar z]$ and $K[\partial, \bar \partial]$ can be identified by sending $z$ to $\partial$ and $\bar z$ to $\bar \partial$. In particular, we may identify $K[\partial \bar \partial]$ as a subring of $K[\partial, \bar \partial]$ and treat this ring as a module over the former. Thus, the structure theorems that we put forward for $K[z,\bar z]$ in the following sections do in many cases pass on to the ring $K[\partial,\bar \partial]$. 

\bigskip 

\section{Harmonic gradings of the polynomial ring $K[z,\bar z]$}

The subring $K^{\star}=K[|z|^2]$ of $K[z,\bar z]$ can be identified with the ring $K[x]$ in one indeterminant and contains the element      
\begin{equation}
\label{binom}
(1-|z|^2)^n=\sum_{j=0}^{n}\binom{n}{j}|z|^{2j}(-1)^{j}, \quad n \in \N.
\end{equation} 
The elements in \eqref{binom} are linearly independent over $K$ and spans $K[|z|^2]$. On another note,
\begin{equation*}
(1-|z|^2)^n (1-|z|^2)^m=(1-|z|^2)^{n+m}, \quad n,m \in \N.
\end{equation*} 
This gives $K[|z|^2]$ the structure of a graded ring 
\begin{equation}
\label{gradedringonevariable}
K[|z|^2] = \bigoplus_{n=0}^{\infty} K \cdot (1-|z|^2)^n.
\end{equation} 
The polynomial ring $K[z,\bar{z}]$ also carries a grading as a module over this ring, that is compliant with \eqref{gradedringonevariable} and that expresses this ring in terms of certain generalized harmonic polynomials or functions. We will give it towards the end of the current section, following a few preliminary notions and some basic structural accounts of the polynomial ring $K[z,\bar z]$.

In relation to the operator,  
\begin{equation}
\label{angularder}
A=z\partial-\bar{z} \bar \partial \in \mathfrak{H}_2 \subset A_2(K),
\end{equation}
we introduce the collection of projection operators
\begin{equation}
\label{projectionmaps}
\pi_j^S=\prod_{\substack{k=1 \\ k \neq j}}^n
\frac{A-m_k}{m_j-m_k} \in \mathfrak{H}_2(K),
\end{equation} 
for a finite collection $S$ of distinct integers $m_1,m_2,\ldots,m_n \in \Z$. We will also write $\pi_j=\pi_j^S$ when the underlying set of integers is clear from context. In the fairly common setting of smooth functions on the complex domain, where the complex or Wirtinger derivatives serve to replace $\partial$ and $\bar \partial$, the operator $A$ is also referred to as the angular derivative \cite{OLip,OW2}. Its name stems from the corresponding form that this operator takes in a shift to polar coordinates.  

We will also reserve the symbol $\xi$ for exponentiation of the polynomials $z$ and $\bar z$, and define 
\begin{equation}
\label{xisymbol}
\xi_m(z) = \left\{ 
\begin{array}{ccc}
z^m, &\quad m \geq 0 \\

\bar{z}^{|m|}, &\quad m < 0.  
\end{array}\right.
\end{equation}
In regard to the current setting, there is a slight abuse of language involved here. We will stick with this notation however, while being both convenient and well adapted to later settings. 

A straightforward calculation shows that the operator $A$ in \eqref{angularder} commutes with elements in the ring $K[|z|^2] \subset A_2(K)$, and that
\begin{equation}
\label{angderactionhom}
Ap(|z|^2)\xi_m(z)=mp(|z|^2)\xi_m(z),
\end{equation} 
for polynomials $p(|z|^2) \in K[|z|^2]$ and integers $m \in \Z$. In particular,  
\begin{equation*}
\pi_{i}p(|z|^2)\xi_{m_j}(z) = \delta_{ij}p(|z|^2)\xi_{m_j}(z), \quad 1 \leq i,j \leq n.
\end{equation*}
For the next statement, we introduce the polynomials  
\begin{equation}
\label{bases}
e_{m,n}(z,\bar{z})=(1-|z|^2)^nz^m,
\end{equation}
and set $e_{-m,n}(z,\bar{z})=\bar{e}_{m,n}(z,\bar{z})=e_{m,n}(\bar{z},z)$ for $m,n \in \N$ in accordance with \eqref{xisymbol}, and write $\mathcal{E}$ for their union. We will refer to the symbols $e_{m,n}$ and $\xi_m$ frequently throughout this text. 
 
\begin{lemma}
\label{lindependence.k}
The set $\mathcal{E}$ is linearly independent over $K$. 
\end{lemma}
\begin{proof}
Let $(m_i,n_i)$ be $k$ distinct sets of pairs in $\Z \times \N$ for $1 \leq i \leq k$. Suppose that 
\begin{equation*}
0=c_{m_1,n_1}e_{m_1,n_1}(z,\bar z)+c_{m_2,n_2}e_{m_2,n_2}(z,\bar z)+\ldots+c_{m_k,n_k}e_{m_k,n_k}(z,\bar z),
\end{equation*}
for some $c_{m_i,n_i} \in  K$. Let $S$ be the set of such integers $m_i$ for $1 \leq i \leq k$, and let $m_l \in S$. By applying the projection operator $\pi_l=\pi_l^S$ to both sides of the last expression, we get that
\begin{equation*}
0=\big(\delta_{m_l,m_1}c_{m_1,n_1}(1-|z|^2)^{n_1}+\ldots+\delta_{m_l,m_k}c_{m_k,n_k}(1-|z|^2)^{n_k}\big)\xi_{m_l}(z).
\end{equation*}
The rightmost factor in this last expression may be cancelled, and
\begin{equation*}
0=\delta_{m_l,m_1}c_{m_1,n_1}(1-|z|^2)^{n_1}+\ldots+\delta_{m_l,m_k}c_{m_k,n_k}(1-|z|^2)^{n_k}.
\end{equation*} 
Since these last terms are linearly independent over $K$, we get in particular that $c_{m_l,n_l}=0$. As $m_l \in S$ was arbitrary, we see that $c_{m_1,n_1}=\ldots=c_{m_k,n_k}=0$.  
\end{proof}

For the next lemma, note that $\xi_k(z) \xi_{-l}(z)=z^k \bar z^l=|z|^{2l}\xi_{k-l}(z)$ for $k \geq l \geq 0$. 

\begin{lemma}
\label{span}
The set $\mathcal{E}$ spans $K[z,\bar{z}]$. 
\end{lemma}
\begin{proof}
It suffices to note that 
\begin{equation*}
|z|^{2l}=(-1)^l(1-|z|^2-1)^l=\sum_{j=0}^{l} \binom{l}{j}(-1)^j(1-|z|^2)^j, \quad l \in \N,
\end{equation*}
and that each $p(z,\bar{z}) \in K[z,\bar{z}]$ is a finite linear sum of terms of the form $|z|^{2l}\xi_m(z)$ for some $l \in \N$ and $m \in \Z$.
\end{proof}

The identification of $K^{\star}=K[|z|^2]$ as a subring of $K[z,\bar{z}]$ realizes the last as a module over $K^{\star}$. The last two lemmas also show that the homogeneous terms $z^m$ together with their conjugates $\bar z^m$ form a set of free generators for this module.
\begin{prop}
\label{linearindependencehomparts}
The homogeneous parts $z^m$ and $\bar{z}^m$ form a set of free generators for the $K^{\star}$-module $K[z,\bar z]$. 
\end{prop}   
\begin{proof}
We know from the previous Lemma \ref{span} that the union of such terms spans $K[z,\bar z]$ over $K^{\star}$. It remains to show that they are linearly independent over $K^{\star}$. 

We can express any polynomial $p(|z|^2)$ in $K[|z|^2]$ as 
\begin{equation*}
\label{polexpbinomscor}
p(|z|^2)=c_{1}(1-|z|^2)^{n_{1}}+\ldots+c_{k}(1-|z|^2)^{n_{k}},
\end{equation*}
for some distinct $n_{1},\ldots,n_{k} \in \N$ and $c_{1},\ldots,c_{k} \in K$. Thus, and for this particular polynomial $p \in K^{\star}$, we can write
\begin{equation*}
p \xi_{m}(z)=c_{1}e_{m,n_1}(z,\bar{z})+\ldots+c_{k}e_{m,n_k}(z,\bar{z}),
\end{equation*}
with $e_{m,n}(z,\bar z) \in K[z,\bar{z}]$ as in \eqref{bases} and $\xi_m(z)$ as in \eqref{xisymbol}. Let $m_1,\ldots,m_n \in \Z$ be distinct and suppose that 
\begin{equation*}
0=p_1\xi_{m_1}(z)+p_2\xi_{m_2}(z)+\ldots+p_k\xi_{m_k}(z),
\end{equation*}    
for some $p_1,\ldots,p_k \in K^{\star}$. In consideration of that above, the right hand side of this last expression is a sum of terms of the form $c_{i_{j}} \cdot e_{m_{i},n_{i_j}}(z,\bar z) \in K[z,\bar{z}]$ for some $c_{i_{j}} \in K$ and mutually distinct pairs $(m_{i},n_{i_j})$ in $\Z \times \N$. Since such elements are linearly independent over $K$, all the constants $c_{i_{j}}$ must be zero, and consequently, the same must be true of the elements $p_i \in K^{\star}$ for $1 \leq i \leq k$. 
\end{proof} 

The last result shows that any element in $K[z,\bar z]$ identifies with its coordinates in the ring $K^{\star}=K[|z|^2]$. 

\begin{cor}
Let $m_j \in \Z$ for $1 \leq j \leq n$ and suppose that there are $p_{m_j} \in K^{\star}$ such that
\begin{equation*}
p_{m_1}\xi_{m_1}(z)+p_{m_2}\xi_{m_2}(z)+ \ldots + p_{m_n}\xi_{m_n}(z)=0.
\end{equation*}
Then $p_{m_j}=0$ for $1 \leq j \leq n$.
\begin{proof}
This follows from Proposition \ref{linearindependencehomparts}. 
\end{proof}
\end{cor}

We will also refer to the $K^{\star}$-module $K[z,\bar z]$ as the module $\mathcal{H}$, where $K[z,\bar z]$ is to be taken as the underlying set with multiplication from the left by elements in the ring $K^{\star}=K[|z|^2]$. As we have done on a few occasions already, we will also denote the image of $p(z,\bar z) \mapsto q(|z|^2)p(z,\bar z)$ by $qp(z,\bar z)$ for $q \in K^{\star}$ and $p(z,\bar z) \in K[z,\bar z]$.
\bigskip

The algebra $K[z,\bar z]$ also carries a positive grading that will provide a good overview of that which is to come, and that is given in terms of the finite dimen\-sional $K$-vector spaces 

\begin{equation}
\label{homogeneousterms}
\mathcal{G}_k=\bigoplus_{|m|+n=k} K \cdot e_{m,n}(z,\bar{z}).
\end{equation}
By the previous two lemmas, we can then write 
\begin{equation}
\label{modulegrading}
K[z,\bar z]=\bigoplus_{k=0}^{\infty} \mathcal{G}_k. 
\end{equation}
This makes $K[z,\bar z]$ into a graded module over the graded $K$-algebra $K[|z|^2]$ in \eqref{gradedringonevariable} that respects the grading of $K[|z|^2]$ in the sense that $(1-|z|^2)^n\mathcal{G}_k \subset \mathcal{G}_{n+k}$. The components of this grading will later be interpreted as certain spaces of generalized harmonic polynomials, to be defined over the coming sections. It can also provide some guidance to list the spanning elements of each component in the form of an inverted triangle. 

\section{The invariant action of $\mathfrak{R}_2$ on $\mathcal{H}$}

We saw in equation \eqref{angderactionhom} that the operator $A \in A_2(K)$ in \eqref{angularder} has a simple description relative to the $K^{\star}$-module $K[z,\bar z]$ or $\mathcal{H}$ as multiplication by numbers. If we write $T_{m,A}$ for the ordinary differential operator that takes a polynomial $f(x) \in K[x]$ to $mf(x) \in K[x]$ for $m \in \Z$, then \eqref{angderactionhom} says that 
\begin{equation*}
Ap(|z|^2)\xi_m(z)=\xi_m(z)(T_{m,A}p)(|z|^2)=\xi_m(z)mp(|z|^2), \quad m \in \Z, \quad p\in K^{\star}. 
\end{equation*}
Other examples like $\partial \bar \partial,\partial \bar \partial+z\partial+\bar z \bar \partial+1$ and $(1-|z|^2)\partial \bar \partial + \bar z \bar \partial$, that have appeared in connection with generalized harmonic functions, also show that the span of each of the generators $\xi_m(z)$ for the module $\mathcal{H}$ is left invariant under the action of such operators, and that their action relative to each of these submodules can be described in terms of a corresponding action by an ordinary differential operator on the polynomials in $K^{\star}$. In more precise terms, we ought to demand of a ``generalized harmonic operator`` $D$ that there exists a family of ordinary differential operators $T_{m,D}$ in the first Weyl-algebra $A_1$ over $K$, such that the following holds:
\begin{equation}
\label{homconditionR2}
Dp(|z|^2)\xi_{m}(z)=\xi_{m}(z)T_{m,D}p(|z|^2), \quad m \in \Z, \quad p \in K^{\star}, 
\end{equation} 
where $T_{m,D}p(|z|^2)$ means $(T_{m,D}p)(|z|^2)$.\footnote{The $K$-algebra $A_1$ has as its generators the operators $x$ and $d/dx$, and is a $K$-subalgebra of the endomorphism class on polynomials in $K[x]$.} An operator $D$ that satisfies \eqref{homconditionR2} must commute with rotations. Indeed, the module $\mathcal{H}$ is generated by the collection of elements $\{\xi_m(z)\}_{m=-\infty}^{\infty}$ over $K^{\star}$ and  
\begin{align}
\label{onlyifactionhomparts}
R_{\theta}D(p(|z|^2)\xi_m(z))&=R_{\theta}(T_{m,D}p(|z|^2)\xi_m(z))=e^{im\theta}T_{m,D}p(|z|^2)\xi_m(z) \\ \notag &=D(e^{im\theta}p(|z|^2)\xi_m(z))=DR_{\theta}p(|z|^2)\xi_m(z),
\end{align} 
for all $m \in \Z$ and $\theta \in [0,2\pi)$. It requires some more work to show that the converse is true, namely, that any operator that commutes with rotations must satisfy \eqref{homconditionR2}. Corollary \ref{basisrotgeneralsums} will be very useful in this respect. 

In this aim, we recall the Leibniz rule for differentiation
\begin{equation*}
d^j/dx^j f(x)g(x) = \sum_{i=0}^{j} \binom{j}{i} f^{(j-i)}(x)g^{(i)}(x).
\end{equation*} 
For the next lemma, we also recall the rising factorial or Pochhammer symbol 
\begin{equation*}
(a)_n=a(a+1)\ldots(a+n-1), 
\end{equation*}
for $a \in K$ and $n \in \N$. 
\begin{lemma}
\label{mthpartpartials}
Let $p(z,\bar z)$ be a polynomial in $K[z,\bar z]$ and let $q_m(z,\bar z)=\varphi_m\xi_m(z)$ be its projection onto the $K^{\star}$-submodule of $\mathcal{H}$ that is spanned by the single element $\xi_m(z)$ for $m \in \Z$. Set 
\begin{equation}
\label{coeffmthpartder}
\rho_m^{k,l}(|z|^2)=|z|^{2k}\varphi_m^{(l)}(|z|^2), \quad k,l \geq 0,
\end{equation}
for $m \in \Z$. Then
\begin{equation}
\label{mthpartder}
z^n\partial^{n}q_m(z,\bar z)=\sum_{l=0}^{n}\binom{n}{l}(m-l+1)_l\rho_{m}^{n-l,n-l}\xi_m(z),
\end{equation}
for $m,n \geq 0$, and
\begin{equation}
\label{negmthpartder}
z^n\partial^{n}q_m(z,\bar z)=\rho_m^{n,n}\xi_m(z),
\end{equation}
for $n \geq 0$ and $m < 0$. 
\end{lemma}
\begin{proof}
Let $m,n \in \N$ be non-negative. Then by Leibniz rule for differentiation, 
\begin{align}
\label{partial^nlemmainvariance}
\partial^n\varphi_m(|z|^2)\xi_m(z) &=\sum_{l=0}^{n}\binom{n}{l}\partial^{n-l}\varphi_m(|z|^2) \partial^l\xi_m(z) \\ \notag &=\sum_{l=0}^{n}\binom{n}{l}\xi_{l-n}(z)\varphi_m^{(n-l)}(|z|^2) (m-l+1)_l\xi_{m-l}(z).
\end{align}
The identity in \eqref{mthpartder} is then obtained by multiplying each of the two sides by $\xi_n(z)=z^n$, noting that $\xi_{k}(z)\xi_{-k}(z)=|z|^{2k}$ for $k \geq 0$. 

As for the identity in \eqref{negmthpartder}, we simply note that 
\begin{equation*}
\partial^n\varphi_{m}(|z|^2)\xi_{m}(z)=\xi_{m}(z)\partial^n\varphi_{m}(|z|^2)=\xi_{-n}(z)\xi_{m}(z)\varphi_m^{(n)}(|z|^2),
\end{equation*} 
for $n \geq 0$ and $m < 0$. 
\end{proof}
The last result shows that the span of each $\xi_m(z)$ over $K^{\star}$ is invariant under $z^n\partial^n$ for each $n \geq 0$. The corresponding result for the operators $\bar z^n \bar \partial^n$ is obtained via conjugation.

\begin{lemma}
\label{mthpartlaplacepowers}
Let $p(z,\bar z)$ be a polynomial in $K[z,\bar z]$ and let $q_m(z,\bar z)=\varphi_m\xi_m(z)$ be its projection onto the $K^{\star}$-submodule of $\mathcal{H}$ that is spanned by the single element $\xi_m(z)$ for $ m \in \Z$. Let $\rho_m^{k,l} \in K^{\star}$ be as in \eqref{coeffmthpartder} for $k,l\geq 0$ and $m \in \Z$. Then
\begin{equation}
\label{laplacepowermthhompol}
\partial^n\bar \partial^n q_m(z,\bar z)=\sum_{l=0}^{n}\binom{n}{l}\frac{(|m|+n)!}{(|m|+n-l)!}\rho_{m}^{n-l,2n-l}\xi_m(z),
\end{equation}
for $n \in \N$ and $m \in \Z$. 
\end{lemma}
\begin{proof}
Let $m,n \in \N$. Then 
\begin{align*}
\partial^n\bar \partial^n q_m(z,\bar z) &= \partial^n \xi_{m+n}(z)\varphi_m^{(n)}(|z|^2) \\ &=\sum_{l=0}^{n}\binom{n}{l} \partial^{n-l}\varphi_m^{(n)}(|z|^2)\partial^l\xi_{m+n}(z) \\ &= \sum_{l=0}^{n}\binom{n}{l}\xi_{l-n}(z)\varphi_m^{(2n-l)}(|z|^2)\frac{(m+n)!}{(m+n-l)!}\xi_{m+n-l}(z) \\ &=\sum_{l=0}^{n}\binom{n}{l}\frac{(m+n)!}{(m+n-l)!}\rho_m^{n-l,2n-l}\xi_m(z). 
\end{align*}
The case for negative $m < 0$ is obtained via conjugation.  
\end{proof}
The last lemma shows that the submodules $K^{\star} \xi_m(z) \subset \mathcal{H}$ are invariant under the action of $\partial^n \bar \partial^n$, for $m \in \Z$ and $n \in \N$. Making use of Corollary \ref{basisrotgeneralsums}, we can now give the converse statement to that which was raised in connection with \eqref{homconditionR2} at the beginning of this section. Namely, that any operator in $\mathfrak{R}_2 \subset A_2$ that commutes with rotations must satisfy \eqref{homconditionR2} with respect to some family of ordinary differential operators $T_{m,D}$ in the first Weyl-algebra $A_1$ over $K$.  
\begin{thm}
\label{mappingA2toA1}
Let $D$ belong to the $K$-subalgebra $\mathfrak{R}_2$ of $A_2(K)$ that consists of all operators in $A_2$ that commute with rotations. Then there exists a family of operators $\{T_{m,D}\}_{m=-\infty}^{\infty}$ contained in $A_1(K)$ such that \eqref{homconditionR2} holds.   
\end{thm}
\begin{proof}
The statement follows from Lemma \ref{mthpartpartials} and Lemma \ref{mthpartlaplacepowers} with the use of Corollary \ref{basisrotgeneralsums}. A slightly more constructive account can be given as follows.

\bigskip

According to Corollary \ref{basisrotgeneralsums}, we can write $D \in \mathfrak{R}_2$ as a sum of terms that are products in $z^{\gamma_1}\bar z^{\gamma_1}, z^{\gamma_2}\partial^{\gamma_2}, \bar z^{\gamma_3} \bar \partial^{\gamma_3}$ and $\partial^{\gamma_4}\bar \partial^{\gamma_4}$ for some numbers $\gamma_1,\gamma_2,\gamma_3,\gamma_4 \in \N$, taken in the order in which these elements were listed. In other words, 
\begin{equation*}
D=\sum_{i=1}^{n}c_iz^{\gamma_{1,i}}\bar z^{\gamma_{1,i}} z^{\gamma_{2,i}}\partial^{\gamma_{2,i}} \bar z^{\gamma_{3,i}} \bar \partial^{\gamma_{3,i}}\partial^{\gamma_{4,i}}\bar \partial^{\gamma_{4,i}} \in \mathfrak{R}_2 \subset A_2(K),
\end{equation*} 
for some numbers $\gamma_{1,i},\gamma_{2,i},\gamma_{3,i},\gamma_{4,i} \in \N$. Let $m \in \N$ and set 
\begin{equation*}
q_{m}(z,\bar z)=\varphi_m\xi_m(z),
\end{equation*}
for some $\varphi_m \in K^{\star}$.  

Let us also fix a number $s \in \N$ in the interval $1 \leq s \leq n$. From Lemma \ref{mthpartlaplacepowers} and equation \eqref{laplacepowermthhompol}, we then see that 
\begin{equation*}
\partial^{\gamma_{4,s}} \bar \partial^{\gamma_{4,s}}q_{m}(z,\bar z)=\xi_m(z)T_{m,\gamma_{4,s}}^{''''}\varphi_m(|z|^2),
\end{equation*}
where $T_{m,\alpha}^{''''} \in A_1(K)$ is the ordinary differential operator given by
\begin{equation*}
T_{m,\alpha}^{''''}=\sum_{l=0}^{\alpha}\binom{\alpha}{l}\frac{(|m|+\alpha)!}{(|m|+\alpha-l)!}x^{\alpha-l}\frac{d^{2\alpha-l}}{dx^{2\alpha-l}}, \quad \alpha \in \N, \quad m \in \Z. 
\end{equation*}
Define
\begin{equation*}
\varphi_3(|z|^2)=T_{m,\gamma_{4,s}}^{''''}\varphi_m(|z|^2).
\end{equation*}
Via conjugation and Lemma \ref{mthpartpartials}, we then get that 
\begin{equation*}
\bar z^{\gamma_{3,s}} \bar \partial^{\gamma_{3,s}}\partial^{\gamma_{4,s}} \bar \partial^{\gamma_{4,s}}q_{m}(z,\bar z)=\xi_m(z)T^{'''}_{m,\gamma_{3,s}}\varphi_3(|z|^2),
\end{equation*}
where $T^{'''}_{m,\alpha} \in A_1(K)$ is the ordinary differential operator 
\begin{equation*}
T^{'''}_{m,\alpha}=x^{\alpha}\frac{d^{\alpha}}{dx^{\alpha}}, \quad \alpha \in \N.
\end{equation*}
Continuing along these lines, we let 
\begin{equation*}
\varphi_2(|z|^2)=T_{m,\gamma_{3,s}}^{'''}T_{m,\gamma_{4,s}}^{''''}\varphi_m(|z|^2).
\end{equation*}
By Lemma \ref{mthpartpartials} and equation \eqref{mthpartder},
\begin{equation*}
z^{\gamma_{2,s}}\partial^{\gamma_{2,s}} \bar z^{\gamma_{3,s}} \bar \partial^{\gamma_{3,s}}\partial^{\gamma_{4,s}}\bar \partial^{\gamma_{4,s}}q_m(z,\bar z)=\xi_m(z)T_{m,\gamma_{2,s}}^{''}\varphi_2(|z|^2),
\end{equation*}
where $T_{m,\alpha}^{''} \in A_1(K)$ is the ordinary differential operator
\begin{equation*}
T_{m,\alpha}^{''}=\sum_{l=0}^{\alpha}\binom{\alpha}{l}(|m|-l+1)_{l}x^{\alpha-l}\frac{d^{\alpha-l}}{dx^{\alpha-l}}, \quad \alpha \in \N, \quad m \in \Z.
\end{equation*}
Finally, we let 
\begin{equation*}
\varphi_1(|z|^2)=T_{m,\gamma_{2,s}}^{''}T_{m,\gamma_{3,s}}^{'''}T_{m,\gamma_{4,s}}^{''''}\varphi_m(|z|^2).
\end{equation*}
Then 
\begin{equation*}
z^{\gamma_{1,s}} \bar z^{\gamma_{1,s}}z^{\gamma_{2,s}} \partial^{\gamma_{2,s}} \bar z^{\gamma_{3,s}} \bar \partial^{\gamma_{3,s}}\partial^{\gamma_{4,s}}\bar \partial^{\gamma_{4,s}}q_m(z,\bar z)=\xi_m(z)T_{m,\gamma_{1,s}}^{'}\varphi_1(|z|^2),
\end{equation*}
where $T_{m,\alpha}^{'}$ is the operator in $A_1$ that acts on elements in $K[x]$ via multiplication by $x^{\alpha}$. 

Take the product of such operators to be
\begin{equation}
\label{TioperatorstheoremR2invariance}
T_i=T_{m,\gamma_{1,i}}^{'}T_{m,\gamma_{2,i}}^{''}T_{m,\gamma_{3,i}}^{'''}T_{m,\gamma_{4,i}}^{''''},
\end{equation} 
for $1 \leq i \leq n$, and write $T_{m,D}=\sum_{i=1}^n c_iT_i$ for their sum. Then
\begin{equation*}
Dq_m(z,\bar z)=\xi_m(z)\sum_{i=1}^{n}c_iT_i\varphi_{m}(|z|^2)=\xi_m(z)T_{m,D}\varphi_{m}(|z|^2),
\end{equation*}
for $m \geq 0$. This completes the case for $m \in \N$. 

Let $m \in \Z \setminus \N$ and let $T_{m,\alpha}^{'},\ldots,T_{m,\alpha}^{''''} \in A_1(K)$ be as above. Set 
\begin{equation*}
\tilde{T}_i=T_{m,\gamma_{1,i}}^{'}T_{m,\gamma_{3,i}}^{''}T_{m,\gamma_{2,i}}^{'''}T_{m,\gamma_{4,i}}^{''''}, \quad 1 \leq i \leq n,
\end{equation*}
and note that $\tilde{T}_i$ is the operator obtained from $T_i$ in \eqref{TioperatorstheoremR2invariance} by interchanging the two parameters $\gamma_{2,i}$ and $\gamma_{3,i}$ for $1 \leq i \leq n$. Take the sum of such products of operators to be $\tilde{T}_{m,D}=\sum_{i=1}^{n} c_i \tilde{T}_i$. The preceding analysis then shows that 
\begin{equation*}
\tilde{D}\tilde{q}_m(z,\bar z)=\bar{\xi}_{m}(z)\tilde{T}_{m,D}\varphi_m(|z|^2)=\xi_{-m}(z)\tilde{T}_{m,D}\varphi_m(|z|^2),
\end{equation*}
for $m \in \Z \setminus \N$, where $\tilde{q}_m(z,\bar z)=\bar{q}_m(z,\bar z)$ and $\tilde{D}=\bar {D}$, and with conjugation taken in the sense of section two. By conjugating each side above, we obtain $Dq_m(z,\bar z)=\xi_{m}(z)\tilde{T}_{m,D}\varphi_m(|z|^2)$ for $m \in \Z \setminus \N$. 
\end{proof}
\begin{cor}
\label{hompartactioncor}
Let $D \in A_2(K)$. Then the operator $D$ satisfies
\begin{equation*}
Dp(|z|^2)\xi_{m}(z)=\xi_{m}(z)T_{m,D}p(|z|^2), \quad m \in \Z, \quad p \in K^{\star},
\end{equation*}
relative to some subfamily of operators $\{T_{m,D}\}_{m=-\infty}^{\infty} \subset A_1(K)$ if and only if $D$ commutes with rotations.  
\end{cor}
\begin{proof}
We saw in \eqref{onlyifactionhomparts} that the forward direction holds, and the preceding Theorem \ref{mappingA2toA1} establishes the converse of this statement. 
\end{proof}
We can also describe the relationship in \eqref{homconditionR2} between elements of $\mathfrak{R}_2 \subset A_2(K)$ and that of the first Weyl-algebra $A_1(K)$ in terms of mappings between the two. If we fix $m \in \Z$ and define $\Lambda_m:\mathfrak{R}_2 \rightarrow A_1$ to be the map that takes the operator $D \in \mathfrak{R}_2$ to the element $T_{m,D} \in A_1$ if and only if
\begin{equation}
\label{mapR2A1}
Dp(|z|^2)\xi_{m}(z)=\xi_{m}(z)T_{m,D}p(|z|^2), \quad p \in K^{\star},
\end{equation} 
then $\Lambda_m$ gives a well defined map from $\mathfrak{R}_2$ into $A_1$. One way to see this is to recall that the algebra $A_1$ is the span of the canonical basis elements $x^id^j/dx^j$ for $i,j \geq 0$, from which it follows that the only element that acts trivially on all of $K[x]$ is the zero operator. The more precise description of what these mappings are or how they evaluate with respect to certain subsets of $\mathfrak{R}_2 \subset A_2(K)$ has merit on its own. A more detailed such analysis was for example key in deriving the series representations for the generalized harmonic functions that were considered in \cite{K}. These ideas are contained in the following geometric interpretation of Theorem \ref{mappingA2toA1}, for which we recall that $p:=p(z,\bar z) \in K[z,\bar z]$ is in the common solution set of $D_1,\ldots,D_n \in A_2(K)$ if and only if 
\begin{equation*}
D_1(p)=\ldots=D_n(p)=0. 
\end{equation*}

\begin{cor}
Let $m \in \Z$ and let $S$ be the set of all operators $D \in \mathfrak{R}_2$ that satisfy 
\begin{equation*}
Dp(|z|^2)\xi_{m}(z)=\xi_{m}(z)T_{m}p(|z|^2), \quad p \in K^{\star},
\end{equation*}
for some given $T_{m} \in A_1(K)$. Let $Q$ be the zero set of $T_{m}$ in $K[|z|^2]$. Then $Q\xi_m(z)$ is contained in the common solution set of all elements in $S$. 
\end{cor}
\begin{proof}
By the previous Theorem \ref{mappingA2toA1}, each element $D \in \mathfrak{R}_2$ is given relative to some uniquely determined sequence $\{T_{m,D}\}_{m\in \Z}$ such that \eqref{homconditionR2} holds. 
\end{proof} 

Theorem \ref{mappingA2toA1} also shows that any submodule of $\mathcal{H}$ that is the span of some number of generators $\xi_m(z)$ for $\mathcal{H}$ is invariant under the action of any element $D \in \mathfrak{R}_2$. 

\begin{cor}
\label{invariantsubspacescor}
Let $S$ be a finite subset of $\Z$. Let $N_S$ denote the $K^{\star}$-submodule of $\mathcal{H}$ that is spanned by the elements $\xi_{m}(z)$ for $m \in S$. Then $N_S$ is invariant under $\mathfrak{R}_2$.
\end{cor}
\begin{proof}
The statement follows from another application of Theorem \ref{mappingA2toA1}.
\end{proof}
We may add that this last result allows one to add another structural layer on top of the submodules $N_S \subset \mathcal{H}$, and treat them as modules over $\mathfrak{R}_2$ or $\mathfrak{H}_2$.   

Let $D \in \mathfrak{R}_2$ be given relative to some sequence $\{T_{m,D}\}_{m\in \Z} \subset A_1(K)$ such that \eqref{homconditionR2} holds, and let $Z(T_{m,D})$ be the zero set of the ordinary differential operator $T_{m,D} \in A_1(K)$, that is
\begin{equation*}
Z(T_{m,D})=\{p(x) \in K[x]: T_{m,D}p(x)=0\}, \quad m \in \Z.
\end{equation*}
The following result then shows that $Dp(z,\bar z)=0$ if and only if $p(z,\bar z) \in K[z,\bar z]$ is contained in the $K$-vector direct sum of the subspaces $Z(T_{m,D})\xi_m(z)$. This reduces the two-dimensional problem $Dp(z,\bar z)=0$ to a corresponding ordinary one-dimensional problem, and allows for an analysis of each part separately.  

\begin{cor}
Let $S$ and $N_S$ be as in Corollary \ref{invariantsubspacescor}. Let $D \in \mathfrak{R}_2$ be given relative to some sequence $\{T_{m,D}\}_{m\in \Z} \subset A_1(K)$ such that \eqref{homconditionR2} holds, and consider the polynomial 
\begin{equation*}
p(z,\bar z)=\sum_{m \in S} p_{m}\xi_{m}(z) \in N_S,
\end{equation*}
for some $p_m \in K^{\star}$. Then $Dp(z,\bar z)=0$ if and only if $T_{m,D} p_m =0$ for every $m \in S$.  
\end{cor}    
\begin{proof}
The statement follows from Proposition \ref{linearindependencehomparts} and Theorem \ref{mappingA2toA1}.
\end{proof}

\section{Generalized harmonic polynomials}

The preceding discussion motivates the following definition. 

\begin{dfn}
A generalized harmonic polynomial is a polynomial $p(z,\bar z)$ in $K[z,\bar z]$ that satisfies $Dp(z,\bar z)=0$ for some $D \in \mathfrak{H}_2$.
\end{dfn}

We will sometimes also refer to the polynomial $p(z,\bar z)$ in the space $K[z,\bar z]$ that is annihilated by or is associated with a given operator $D$ in $\mathfrak{H}_2$, and refer to this relation as a pair $(p,D) \in \mathcal{H} \times \mathfrak{H}_2$, where we recall the use of the symbol $\mathcal{H}$ as a replacement for the $K^{\star}$-module $K[z,\bar z]$. In abuse of the language, we will sometimes also refer to the polynomial $(p,D)$ when speaking in broader terms, or when we want to emphasize that $p(z,\bar z)$ is a polynomial in $K[z,\bar z]$ that satisfies $Dp(z,\bar z)=0$.

In the restricted case of polynomials it is clear to see that the collection of all such generalized harmonic polynomials is $K[z,\bar z]$ itself, since every polynomial $p(z,\bar z)$ satisfies $\partial^n\bar \partial^n p=0$ for some $n \in \N$. The question becomes more delicate indeed, when for example $K[z,\bar z]$ is taken to be a mere dense subset of some larger space, for example $C^{\infty}(V)$ for some compact subset $V$ of $\D$,  or when considering the common solution set for some $D_1,\ldots,D_n \in \mathfrak{H}_2$. 

The coming part of this section concerns the various ways in which a polynomial $(p,P)$ can be represented in terms of other polynomials $(q_i,Q_i)$ for a given collection of operators $P, Q_i \in \mathfrak{H}_2$. Special attention will be drawn to the case where $P=\partial \bar \partial$ and $p$ is a polynomial satisfying $P^n p=0$ for some positive $n \in \N$, and where the $(q_i,Q_i)$ are the so called $(\gamma_1,\gamma_2)$-harmonic polynomials, to be defined over the next few paragraphs. The reverse relationship between such generalized harmonic polynomials or functions will also be of interest.  

Note that $K[z,\bar{z}]$ is the smallest algebra that contains $\mathcal{H}^1$, the set of polynomials in $K[z,\bar{z}]$ of the form $\sum_{m} c_mz^m+c_{-m}\bar{z}^m$. We will say that a polynomial $p(z,\bar{z})$ in $K[z,\bar{z}]$ is harmonic when $\partial\bar{\partial}p(z,\bar{z})=0$. In this terminology, every polynomial in $\mathcal{H}^1$ is harmonic. More generally, we will refer to any polynomial in $K[z,\bar{z}]$ that satisfies the equation $\partial^n\bar{\partial}^np(z,\bar{z})=0$ as a polyharmonic polynomial of order $n \in \N$ and denote the collection of all such polynomials by $\mathcal{H}^n$. We will sometimes also refer to such a polynomial as an $n$-harmonic polynomial. The latter use of terminology is perhaps to be preferred over the former, and avoids the possible misinterpretation that the prefix is somehow related to the degree of the concerned polynomial. We will stick with this terminology however, and rely on context for its interpretation, noting that there are polyharmonic polynomials of order $n$, or $n$-harmonic polynomials, to any degree for $n \geq 1$. In the case that $n=0$, we will simply define $\mathcal{H}^0$ to be the zero set. The set $\mathcal{H}^n$ forms a $K$-vector subspace in $K[z,\bar{z}]$ for all $n \in \N$. Clearly, any polynomial $p(z,\bar{z}) \in K[z,\bar{z}]$ is $N$-harmonic for some least integer $N \in \N$, in which case it is also $n$-harmonic for all $n \geq N$. Another way to state this is in terms of the ascending sequence $\{0\} = \mathcal{H}^0 \subset \mathcal{H}^1 \subset \mathcal{H}^2 \subset\ldots$ whose union is $\mathcal{H}=K[z,\bar{z}]$. Alternatively, the $K$-vector subspaces $\mathcal{H}^n$ form a filtration of $K[z,\bar z]$ as a vector space over $K$. In fact,
\begin{equation*}
\mathcal{H}^n=h_1 \oplus h_2 \oplus \ldots \oplus h_{n}, \quad n \geq 1,
\end{equation*} 
where $h_{n+1}$ is the span of $|z|^{2n} \cdot \mathcal{H}^1$ for $n \geq 0$, as clarified below. The subspaces $h_{n}$ are thus direct summands of $K[z,\bar z]$ and appear as the ``$n$:th layer`` in the infinite Matryoshka that results by sorting the summands $h_n$ in the form of an inverted triangle, with the generators of $h_{1}=\mathcal{H}^1$ making up the ``outer shell``, followed by the generators of $h_{2}$, and so on. More commonly, such decompositions for polyharmonic polynomials or functions are referred to as those of Almansi, as further explained below. 

In relation to the previous sections, we can further note that the $K$-linear map $\partial \bar \partial$ brings $h_{n}$ injectively onto $h_{n-1}$ while the operators $z\partial$ and $\bar z \bar \partial$ bring $h_n$ onto itself for $n \geq 2$. That such operators fail to be injective in the case of $n=1$ rests in the observation that they annihilate constants and that the polynomials in $\mathcal{H}^1$ are harmonic. The analysis of how the solution spaces to members in $\mathfrak{R}_2$ or $\mathfrak{H}_2$ are formed must then come down to a study of how the elements in $K[z,\bar z]$ are mapped between the various components under such operators, of which we will see plenty examples. 

In summary then, we have $h_1=\mathcal{H}^1$ and the split exact sequence
\[
  \setlength{\arraycolsep}{1pt}
  \begin{array}{*{9}c}
    0 & \Lrightarrow & \mathcal{H}^{n-1} & \overset{id}  \Lrightarrow & \mathcal{H}^n & \overset{\partial^{n-1} \bar \partial^{n-1}}\Lrightarrow & \mathcal{H}^{1} & \Lrightarrow & 0
  \end{array}, 
\]
for $n \geq 2$. We can also draw the diagram
\begin{center}
\begin{tikzpicture}
    \node (E) at (0,0) {$\mathcal{H}^n$};
    \node[right=of E] (F) {$\mathcal{H}^1$};
    \node[below=of F] (A) {$h_n,$};
    \node[below=of E] (Asubt) {$\frac{\mathcal{H}^n}{\mathcal{H}^{n-1}}$};
    \draw[->] (E)--(F) node [midway,above] {};
    \draw[->] (F)--(A) node [midway,right] {} 
                node [midway,left] {};
    \draw[->] (Asubt)--(A) node [midway,below] {} 
                node [midway,above] {};
    \draw[->] (E)--(Asubt) node [midway,left] {};
    \draw[->] (Asubt)--(F);
\end{tikzpicture}
\end{center}
and note from the correspondence
\begin{equation*}
\mathcal{H}^n=h_1 \oplus \ldots \oplus h_n \cong \mathcal{H}^1/\mathcal{H}^{0} \times \ldots \times \mathcal{H}^{n}/\mathcal{H}^{n-1} \cong \underbrace{\mathcal{H}^1 \times \ldots \times \mathcal{H}^1}_{n \ \text{times}},
\end{equation*}
that $\mathcal{H}^n$ is indeed the $n$-fold product of $\mathcal{H}^1$ with itself. The filtering components $\mathcal{H}^n$ can also be linked by the chain 
\[
  \setlength{\arraycolsep}{1pt}
  \begin{array}{*{11}c}
   \ldots & \overset{\partial^{n}\bar \partial^{n}} \Lrightarrow & \mathcal{H}^{n} & \overset{\partial^{n-1}\bar \partial^{n-1}}  \Lrightarrow & \mathcal{H}^{n-1} & \overset{\partial^{n-2} \bar \partial^{n-2}} \Lrightarrow & \ldots & \overset{\partial \bar \partial} \Lrightarrow & \mathcal{H}^1 & \Lrightarrow & 0,
  \end{array} 
\]
with corresponding homology groups  
\begin{equation*}
H_n=\mathcal{H}^{n}/\mathcal{H}^1 \cong h_2 \oplus \ldots \oplus h_n \cong \mathcal{H}^{n-1}, \quad n \geq 1,
\end{equation*} 
measuring the extent to which the polyharmonic components fail to be harmonic. 

We shall proceed with a more direct approach and show how these observations manifest in computational terms.  
 
\begin{lemma}
\label{polyharmonicorder}
A polynomial $p(z,\bar{z}) \in \mathcal{H}^n$ is polyharmonic of order $n \in \N$ if and only if $p(z,\bar{z})$ is a linear combination of terms of the form $z^k \bar{z}^l$, where $k,l \in \N$ are such that $k < n$ or $l < n$.
\end{lemma}
\begin{proof}
Let $q_{k,l}(z,\bar{z})=z^k \bar z^l$ be such that $k \geq l \geq n$. Then the operator $\partial^n\bar{\partial}^n$ applied to $q_{k,l}(z,\bar{z})$ gives a polynomial in $K[z,\bar{z}]$ of the form $|z|^{2N}z^{M}$, where $M=k-l$ and $N=l-n$. By taking conjugates and interchanging $k$ for $l$ we get a similar expression in the case that $l > k \geq n$. 

Suppose that $p(z,\bar z) \in K[z,\bar z]$ is a polyharmonic polynomial of order $n \in \N$ and write
\begin{equation*}
p(z,\bar z)=\sum_{k}\sum_{l} c_{k,l}q_{k,l}(z,\bar z),
\end{equation*}
where $c_{k,l}=0$ for all but finitely many indices $k,l\geq 0$. Then 
\begin{equation*}
0=\partial^n \bar \partial^n p(z,\bar z)=\sum_{k\geq l \geq n} p_{k,l}\xi_{k-l}(z)+\sum_{l > k \geq n} p_{k,l}\xi_{k-l}(z),
\end{equation*} 
where $p_{k,l} \in K^{\star}$ is such that 
\begin{equation*}
p_{k,l}(|z|^2) = \left\{ 
\begin{array}{ccc}
c_{k,l}'|z|^{2(l-n)}, &\quad k \geq l \geq n, \\
c_{k,l}'|z|^{2(k-n)}, &\quad l > k \geq n,  
\end{array}\right.
\end{equation*}
for some $c_{k,l}' \in K$. Since $\{\xi_m(z)\}_{m=-\infty}^{\infty}$ is a free generating set for $\mathcal{H}$ over $K^{\star}$, we get that $p_{k,l}=0$ whenever $k,l \geq n$. The converse is clear. 
\end{proof}

Another way of phrasing the previous is to say that  
\begin{equation}
\label{spanharmoniccomp}
\mathcal{H}^n=\textnormal{span}_K\{z^i \bar z^j: i < n \ \textnormal{or} \ j < n \}, \quad n \geq 0,
\end{equation}
with the interpretation that $\mathcal{H}^0=\{0\}$. It can also be noted that, while the filtration $\{\mathcal{F}_n\}_{n=0}^{\infty}$ for $K[z,\bar z]$ that is induced from the grading in \eqref{modulegrading} is well behaved under multiplication, the same cannot be said of the polyharmonic components $\mathcal{H}^n$. By this we mean that the components $\mathcal{F}_n=\bigoplus_{k=0}^n \mathcal{G}_k$ satisfy $\mathcal{F}_m\mathcal{F}_n \subset \mathcal{F}_{m+n}$, while $\mathcal{H}^m \mathcal{H}^n$ is not contained in $\mathcal{H}^{m+n}$. In fact, the multiplication of polyharmonic components is not restrictive at all, and any two polyharmonic components $\mathcal{H}^m,\mathcal{H}^n$ reproduces $K[z,\bar z]$ for $m,n \geq 1$. That is, provided with any of the basis elements $z^i \bar z^j$ of $K[z,\bar z]$ for some $i,j \geq 0$, there is an element in $\mathcal{H}^m$ and one in $\mathcal{H}^n$ such that their product is $z^i \bar z^j$. Indeed, we can take $z^i \in \mathcal{H}^m$ and $\bar z^j \in \mathcal{H}^n$ for $m,n \geq 1$. Effectively, this means that for any given $N \geq 1$, it is possible to pick polyharmonic polynomials $p_m(z,\bar z) \in \mathcal{H}^m$ and $p_n(z,\bar z) \in \mathcal{H}^n$ such that their product satisfies $\partial^N\bar \partial^N (p_mp_n)(z,\bar z) \neq 0$.

We shall continue with the following so called Almansi representation for poly\-harmonic polynomials. 
\begin{prop}
\label{almansi}
Let $p(z,\bar{z}) \in \mathcal{H}^n$ be a polyharmonic polynomial of order $n$. Then $p(z,\bar{z})$ is in the linear span of $\mathcal{H}^{1}$ over $K^{\star}$. In particular, we can write 
\begin{equation*}
p(z,\bar{z})=q_0(z,\bar{z})+(1-|z|^2)q_1(z,\bar{z})+\ldots+(1-|z|^2)^{n-1}q_{n-1}(z,\bar{z}),
\end{equation*}
for some polynomials $q_0(z,\bar z),\ldots,q_{n-1}(z,\bar{z}) \in \mathcal{H}^1$.
\end{prop}
\begin{proof}
By the preceding Lemma \ref{polyharmonicorder}, any $p(z,\bar{z}) \in \mathcal{H}^n$ is a linear combination of terms of the form $z^k\bar{z}^l$ where either $k < n$ or $l < n$. In the case that $k \geq l$, we may write   
\begin{equation*}
z^k\bar{z}^l=z^{k-l}|z|^{2l}=z^{k-l}\sum_{i=0}^{l}\binom{l}{i}(-1)^i(1-|z|^2)^i.
\end{equation*}
Note in particular that $l \leq n-1$ and that $z^{k-l}$ is harmonic. By taking conjugates and interchanging $k$ for $l$ we get a similar expression in the case that $l \geq k$.  
\end{proof}

A most relevant subspace is given by the polynomials $p(z,\bar{z}) \in K[z,\bar z]$ that satisfy $L_{\gamma_1,\gamma_2}p(z,\bar{z})=0$ with respect to the operator 
\begin{equation}
\label{pqoperator}
L_{\gamma_1,\gamma_2}=(1-|z|^2)\partial \bar{\partial}+\gamma_1 z\partial+\gamma_2 \bar{z} \bar \partial-\gamma_1 \gamma_2 \in \mathfrak{H}_2(K),
\end{equation}
for some $\gamma_1, \gamma_2 \in K$. We will refer to such a polynomial as a $(\gamma_1,\gamma_2)$-harmonic poly\-nomial and denote the corresponding $K$-vector subspace by $\mathcal{H}_{\gamma_1,\gamma_2}$. Note that $p(z,\bar{z})$ is $(\gamma_1,\gamma_2)$-harmonic if and only if its conjugate $\bar{p}(z,\bar z)=p(\bar z, z)$ is $(\gamma_2,\gamma_1)$-harmonic, with conjugation taken in the sense of the first paragraph of section two. Due to their relevance for that which is to come, we can already now mention that a characterization for such polynomials or functions have been achieved, and that their representation in $\mathcal{H}$ is well understood in terms of familiar functions in $K^{\star}$. The result we are referring to was given in \cite[Theorem 5.4]{OK}, and more generally in \cite[Theorem 5.1]{K}. It states that a polynomial $p(z,\bar z) \in K[z,\bar z]$ is $(\gamma_1,\gamma_2)$-harmonic polynomial if and only if 
\begin{align}
\label{pqpolrepresentations}
p(z,\bar z)&=\sum_{m=1}^\infty c_m F(-\gamma_1,m-\gamma_2,m+1;|z|^2)z^m  \\  \notag
& \ +\sum_{m=1}^\infty c_{-m} F(-\gamma_2,m-\gamma_1,m+1;| z |^2)\bar z^m+c_0 F(-\gamma_1,-\gamma_2,1;|z|^2),
\end{align} 
for some numbers $c_m \in K$ that are non-zero for at most finitely many $m \in \Z$, while the hypergeometric functions 
\begin{equation*}
F(a,b,c;x)=\sum_{k=0}^{\infty} \frac{(a)_k(b)_k}{(c)_k}\frac{x^{k}}{k!},
\end{equation*}
that appear in this expression are polynomials in $K[|z|^2]$ only if $(\gamma_1,\gamma_2) \in \N \times K$ or $(\gamma_1,\gamma_2) \in K \times \N$. They arise here as a result of the mapping determined through \eqref{homconditionR2} that was described earlier. It brings the operator $L_{\gamma_1,\gamma_2} \in \mathfrak{H}_2 \subset A_2$ to the family of hypergeometric operators 
\begin{equation*}
T_{m,L_{\gamma_1,\gamma_2}}=x(1-x)\frac{d^2}{dx^2}+[|m|+1-(|m|+1-\gamma_1-\gamma_2)x]\frac{d}{dx}+\gamma_1\gamma_2-r_m|m|, 
\end{equation*}  
for $m \in \Z$, where $\{r_m\}_{m=-\infty}^{\infty}$ is the complex sequence given by $r_m=\gamma_1$ for $m \geq 0$ and $r_m=\gamma_2$ for $m < 0$. 

The original form of this statement was given in more general terms, and states that a function $u \in C^2(\D)$ is $(\gamma_1,\gamma_2)$-harmonic if and only if it satisfies \eqref{pqpolrepresentations} for some numbers $c_m \in \C$ subjected to 
\begin{equation}
\label{limsupcoeff}
\limsup_{|m| \rightarrow \infty} |c_m|^{1/|m|}\leq 1.
\end{equation}
Convergence should then be taken as absolute convergence in the space of smooth functions on $\D$. The coefficients in this expansion were also shown to satisfy 
\begin{equation*}
\left\{ 
\begin{array}{ccc}
m!c_m &= \partial^m u(0), \\

m!c_{-m} &= \bar \partial^m u(0),& \quad m \in \N.  
\end{array}\right.
\end{equation*}
We will refer to this result on several occasions, in the confinement to polynomials and the later more general setting of smooth functions on the unit disc. For now though, we will proceed as we have done and restrict the current discussion to polynomials in $K[z,\bar z]$.   
 
As was noted in the previous paragraph, the realm of parameters in $K \times K$ that produces non-trivial (polynomial) solutions to the equation $L_{\gamma_1,\gamma_2}p(z,\bar z)=0$ is rather restricted to the case where either $(\gamma_1,\gamma_2) \in \N \times K$ or $(\gamma_1,\gamma_2) \in K \times \N$. If the parameters are such that $(\gamma_1,\gamma_2) \in \N \times (K \setminus \N)$, then \eqref{pqpolrepresentations} shows that $c_{-m}=0$ for $m > \gamma_1$. In this case, the second of the two sums in \eqref{pqpolrepresentations} can be written as 
\begin{equation*}
\sum_{m=1}^{\gamma_1} c_{-m} F(-\gamma_2,m-\gamma_1,m+1;| z |^2)\bar z^m=\sum_{m=1}^{\gamma_1}\sum_{k=0}^{\gamma_1-m} c_{-m} \frac{(-\gamma_2)_k(m-\gamma_1)_k}{(m+1)_k}\frac{|z|^{2k}}{k!}\bar z^m.
\end{equation*}    
If we now employ the identity $|z|^{2k}=\sum_{j=0}^{k}\binom{k}{j}(-1)^j(1-|z|^2)^j$ and rearrange terms, we can rewrite this last expression in the form of 
\begin{equation*}
c'_0q_0(z,\bar z)+c'_1(1-|z|^2)q_1(z,\bar z)+\ldots+c'_{\gamma_1-1}(1-|z|^2)^{\gamma_1-1}q_{\gamma_1-1}(z,\bar z),
\end{equation*} 
for some $c'_j \in K$ and harmonic polynomials $q_j(z,\bar z) \in \mathcal{H}^1$ such that $q_j(z,0) = 0$ for $0 \leq j \leq \gamma_1-1$. This also shows that the second sum in \eqref{pqpolrepresentations} is polyharmonic to an order at least $\gamma_1 \in \N$. Since the remaining two parts in \eqref{pqpolrepresentations} are polyharmonic to an order of $\gamma_1+1 \in \N$, their sum must be polyharmonic to an order at least $\gamma_1+1 \in \N$. In a similar way, we get that the first of the two sums is polyharmonic to an order of $\gamma_2 \in \N$ when $(\gamma_1,\gamma_2) \in (K \setminus \N) \times \N$. The least order to which the $(\gamma_1,\gamma_2)$-harmonic polynomial in \eqref{pqpolrepresentations} is polyharmonic will then also depend on the remaining terms. 

We will argue along similar lines in what follows, but simplify somewhat and restrict our attention to the case $(\gamma_1,\gamma_2) \in \N \times \N$ as we conclude. In the aim for some precision and in order to spare some words, we let $S_{\N} \subset K[x]$ be the multiplicative set
\begin{equation}
\label{localization}
S_{\N}=\{p(x) \in K[x]: p(m) \neq 0, \ m \geq 0\},
\end{equation}
and write $S_{\N}^{-1}K[x]$ for the localization of $K[x]$ by $S_{\N}$. Elements in $S_{\N}^{-1}K[x]$ are thus fractions of polynomials with denominators in $S_{\N}$. It is worth to emphasize, despite its tautological content, that the ring $S_{\N}^{-1}K[x]$ is closed under addition and multiplication. We stress this in relation to some of the less constructive arguments that will be given later on, where the reduction of more complicated such expressions to a single element in $S_{\N}^{-1}K[x]$ will work to our advantage.   
\begin{lemma}
\label{pqtopolylemma}
Let $\gamma_1 \in \N$ and let $\gamma_2 \in K$. Consider the sequence of polynomials 
\begin{equation*}
q_m(z,\bar z)=F(-\gamma_1,m-\gamma_2,m+1;|z|^2)z^m, \quad m \geq 1. 
\end{equation*}
Then there exist $t_0(x),\ldots,t_{\gamma_1}(x) \in S_{\N}^{-1}K[x]$ and a sequence of polynomials 
\begin{equation}
\label{mcomponentspqtopolylemma}
s_{m}(|z|^2)=t_0(m)+t_1(m)(1-|z|^2)+\ldots+t_{\gamma_1}(m)(1-|z|^2)^{\gamma_1}, \quad m \geq 1,
\end{equation}
such that 
\begin{equation}
\label{partialsumpqtopollemma}
q_m(z,\bar z)=s_m(|z|^2)z^m, \quad  m \geq 1.
\end{equation} 
\end{lemma}
\begin{proof}
Let $\gamma_1 \in \N$ and let $\gamma_2 \in K$. Set 
\begin{equation}
\label{hypergeompollemmapqtopoly}
s_m(|z|^2)=F(-\gamma_1,m-\gamma_2,m+1;|z|^2), \quad m \geq 1,
\end{equation}
and recall that
\begin{equation*}
|z|^{2k}=\sum_{j=0}^{k}\binom{k}{j}(-1)^j(1-|z|^2)^j.
\end{equation*}
By employing the last in relation to \eqref{hypergeompollemmapqtopoly}, we can write
\begin{equation*}
s_{m}(|z|^2)=\sum_{k=0}^{\gamma_1}\sum_{j=0}^{k}\frac{(-\gamma_1)_k(m-\gamma_2)_k}{(m+1)_k k!}\binom{k}{j}(-1)^j(1-|z|^2)^j, \quad m \geq 1. 
\end{equation*}
Let 
\begin{equation}
\label{rationalexpressionspqtopolyemma}
t_j(x)=(-1)^j\sum_{k=j}^{\gamma_1}\frac{(-\gamma_1)_k(x-\gamma_2)_k}{(x+1)_k k!}\binom{k}{j} \in S_{\N}^{-1}K[x], \quad 0 \leq j \leq \gamma_1. 
\end{equation}
Then 
\begin{equation*}
s_{m}(|z|^2)=t_0(m)+t_1(m)(1-|z|^2)+\ldots+t_{\gamma_1}(m)(1-|z|^2)^{\gamma_1}, \quad m \geq 1.
\end{equation*}
This completes the proof of the statement. 
\end{proof}
Note that the rational expressions $t_0(x),t_1(x),\ldots,t_{\gamma_1}(x) \in S_{\N}^{-1}K[x]$ in  \eqref{rationalexpressionspqtopolyemma} are uniquely determined by $\gamma_1 \in \N$ and $\gamma_2 \in K$. This can easily be checked by employing Lemma \ref{lindependence.k}.
\begin{cor}
\label{pqtopolycor}
Let $\gamma_1 \in \N$ and $\gamma_2 \in K$. Consider the sequence of polynomials 
\begin{equation}
\label{Ssumpqtopoly}
S_N(z,\bar z)=\sum_{m=1}^{N} c_m F(-\gamma_1,m-\gamma_2,m+1;|z|^2)z^m, \quad N \geq 1,
\end{equation}
for some $c_m \in K$. Then
\begin{equation*}
S_N(z,\bar z)=q_{N,0}(z,\bar z)+ (1-|z|^2) q_{N,1}(z,\bar z)+\ldots+(1-|z|^2)^{\gamma_1} q_{N,\gamma_1}(z,\bar z),
\end{equation*}
for some harmonic polynomials $q_{N,j}(z,\bar z) \in \mathcal{H}^1$ such that $q_{N,j}(0,\bar z)=0$ for all $0 \leq j \leq \gamma_1$. In particular, the sum $S_N(z,\bar z)$ in \eqref{Ssumpqtopoly} is poly\-harmonic to an order at least $\gamma_1+1\in \N$ for all $N \geq 1$.
\end{cor}
\begin{proof}
By Lemma \ref{pqtopolylemma} we can write 
\begin{align*}
S_{N}(z,\bar z)=\sum_{m=1}^{N} c_{m}\big(t_0(m)+(1-|z|^2)t_1(m)+\ldots+(1-|z|^2)^{\gamma_1} t_{\gamma_1}(m)\big)z^m,
\end{align*}
for some $t_0(x),t_1(x),\ldots,t_{\gamma_1}(x) \in S_{\N}^{-1}K[x]$. Let $k_{m,j}=c_mt_j(m) \in K$ and define
\begin{equation*}
q_{N,j}(z,\bar z)=\sum_{m=1}^{N} k_{m,j} z^m \in \mathcal{H}^1, \quad 0 \leq j \leq \gamma_1. 
\end{equation*}
Then, 
\begin{equation*}
S_{N}(z,\bar z)=\sum_{j=1}^{\gamma_1}(1-|z|^2)^jq_{N,j}(z,\bar z).
\end{equation*} 
The conclusion is now evident in view of Lemma \ref{polyharmonicorder}. 
\end{proof}

We can then give the following sufficient criteria in respect of the order to which the $(\gamma_1,\gamma_2)$-harmonic polynomial is poly\-harmonic when $\gamma_1,\gamma_2 \in \N$. 

\begin{cor}
\label{pqtopolyprop}
Let $\gamma_1,\gamma_2 \in \N$. Let $p(z,\bar z) \in K[z,\bar z]$ be a $(\gamma_1,\gamma_2)$-harmonic poly\-nomial. Then $p(z,\bar z)$ is polyharmonic to an order at least $\max(\gamma_1+1,\gamma_2+1) \in \N$. 
\end{cor}
\begin{proof}
Since $p(z,\bar z) \in K[z,\bar z]$ is a $(\gamma_1,\gamma_2)$-harmonic polynomial, we can express it in the form of \eqref{pqpolrepresentations} for some $c_m \in K$ that are non-zero for at most finitely many $m \in \Z$. An application of Corollary \ref{pqtopolycor} to the first sum in \eqref{pqpolrepresentations} shows that this sum is polyharmonic to an order at least $\gamma_1+1 \in \N$. A similar conclusion for the second sum in \eqref{pqpolrepresentations} is obtained by taking conjugates, and shows that it is polyharmonic to an order at least $\gamma_2+1$. The last term in \eqref{pqpolrepresentations} is given by
\begin{equation*}
F(-\gamma_1,-\gamma_2,1;|z|^2)=\sum_{k=0}^{\min(\gamma_1,\gamma_2)}\frac{(-\gamma_1)_k (-\gamma_2)_k}{(m+1)_k}\frac{|z|^{2k}}{k!},
\end{equation*}  
which is polyharmonic to an order of $\min(\gamma_1+1,\gamma_2+1)$.  
\end{proof}

We shall continue with the case where $(\gamma_1,\gamma_2) \in \N \times \Z \cup \Z \times \N$ and turn the question on its head. The restriction to such parameter values includes the case $\gamma_1=\gamma_2=0$, associated with the space $\mathcal{H}_{0,0}=\mathcal{H}^{1}$ of harmonic polynomials, and the two cases where either $\gamma_1$ or $\gamma_2$ are non-negative whole numbers while $\gamma_2=-1$ and $\gamma_1=-1$, respectively. Even more relevant perhaps is the case where $\gamma_1=\gamma_2 \in \N$. We will start with the former, and consider the representational form for polyharmonic polynomials in terms of the spaces $\mathcal{H}_{n,-1}$ and $\mathcal{H}_{-1,n}$ when $n \in \N$ is a non-negative integer. 

In connection with the last, we will also refer to the conjugated set $\bar{U}$ of a subset $U$ of $K[z,\bar z]$, by which we mean the set obtained from $U$ by conjugating each of its elements, i.e. $p(z,\bar z) \in U$ if and only if $\bar{p}(z,\bar z)=p(\bar z,z) \in \bar{U}$. We also introduce the $K$-vector spaces 
\begin{equation}
\label{spanbasesn}
\mathcal{E}_{k,n}=\textnormal{span}_K\{e_{m,n}(z,\bar z): m \geq k\}, \quad k \geq 0, \quad n \in \N,
\end{equation}
with $e_{m,n}(z,\bar z)$ as in \eqref{bases} and $e_{-m,n}(z,\bar z)=\bar{e}_{m,n}(z,\bar z)$. Note that $\bar{\mathcal{E}}_{k,n}$ is the span of the polynomials $e_{m,n}(z,\bar z)$ for which $m \leq -k$ when $k,n \in \N$. Note further that $\mathcal{E}_{k,n} \cap \bar{\mathcal{E}}_{k,n}$ is empty for all $k > 0$, and that this intersection is the span of $(1-|z|^2)^n$ over $K$ when $k=0$, for $n \in \N$. In relation to \eqref{pqpolrepresentations}, we also observe that
\begin{equation*}
e_{m,n}(z,\bar z)=\sum_{k=0}^{n}\frac{(-n)_k(m+1)_k}{(m+1)_k}\frac{|z|^{2k}}{k!}z^m=F(-n,m-(-1),m+1;|z|^2)z^m.
\end{equation*} 
This seals the inclusion  
\begin{equation*}
e_{m,n}(z,\bar z) \in \left\{ 
\begin{array}{ccc}
\mathcal{H}_{n,-1}, &\quad m \geq 0, \\

\mathcal{H}_{-1,n}, &\quad m < 0, 
\end{array}\right.
\end{equation*}
and shows that $L_{n,-1}e_{m,n}(z,\bar z)=0$ or $L_{-1,n}e_{m,n}(z,\bar z)=0$, depending on whether $m \geq 0$ or $m < 0$. We choose to summarize these observations as follows.
\begin{lemma}
\label{decompgenharmpolinteger}
Let $n \in \N$. Let $q(z,\bar{z}) \in \mathcal{H}^1 \subset K[z,\bar z]$ and set 
\begin{equation*}
p(z,\bar{z})=(1-|z|^2)^n q(z,\bar{z}) \in K[z,\bar z].
\end{equation*}
Then $p(z,\bar z)$ is in the $K$-vector direct sum of the spaces $\mathcal{E}_{0,n} \cap \bar{\mathcal{E}}_{0,n}$, $\mathcal{E}_{1,n}$ and $\bar{\mathcal{E}}_{1,n}$. In symbols, 
\begin{equation*}
p(z,\bar{z}) \in \mathcal{E}_{0,n} \cap \bar{\mathcal{E}}_{0,n} \oplus \mathcal{E}_{1,n} \oplus \bar{\mathcal{E}}_{1,n}.
\end{equation*}
Furthermore, $\mathcal{E}_{1,n} \subset \mathcal{H}_{n,-1}$ and $\bar{\mathcal{E}}_{1,n} \subset \mathcal{H}_{-1,n}$, while $\mathcal{E}_{0,n} \cap \bar{\mathcal{E}}_{0,n} \subset \mathcal{H}_{n,-1} \cap \mathcal{H}_{-1,n}$.
\end{lemma}
\begin{proof}
We already know from Lemma \ref{polyharmonicorder} that $q(z,\bar{z})$ can be written in the form $\sum_{m=0}^N c_mz^m+c_{-m}\bar{z}^m$ for some $N \in \N$, where $c_m \in K$ for $-N \leq m \leq N$. By conjugate symmetry and by projecting onto the $K^{\star}$-submodule of $\mathcal{H}$ that is spanned by $z^m$ using the projection operators in \eqref{projectionmaps}, we can assume that $p(z,\bar{z})$ is of the form $p_m(z,\bar{z})=(1-|z|^2)^nz^m=e_{m,n}(z,\bar z)$ for some $m \in \N$. In view of the earlier comments, we may then conclude that
\begin{equation*}
L_{n,-1}p_m(z,\bar{z})=L_{n,-1}e_{m,n}(z,\bar{z})=0,
\end{equation*}
for all $m \in \N$. We can also recall from Lemma \ref{lindependence.k} that the $e_{m,n}(z,\bar z) \in K[z,\bar z]$ are linearly independent over $K$ for $m \in \Z$.
\end{proof}

An application of Lemma \ref{decompgenharmpolinteger} to the Almansi representation for polyharmonic poly\-nomials that was given in Proposition \ref{almansi} yields the following representation for such polynomials, as sums over certain $(\gamma_1,\gamma_2)$-harmonic polynomials.  
\begin{prop}
\label{polyharmonicinclusiongenharmpol}
The space $\mathcal{H}^n$ of polyharmonic polynomials of order $n \in \N$ is contained in the $K$-vector direct sum, 
\begin{equation*}
\mathcal{H}^n \subset \ \bigoplus_{j=0}^{n-1} \mathcal{E}_{1,j} \ \bigoplus_{j=0}^{n-1}  \bar{\mathcal{E}}_{1,j} \bigoplus_{j=0}^{n-1} \mathcal{E}_{0,j} \cap \bar{\mathcal{E}}_{0,j},
\end{equation*} 
where $\mathcal{E}_{k,n}$ are the subspaces of $(n,-1)$-harmonic polynomials in \eqref{spanbasesn}. 
\end{prop} 
\begin{proof}
Let $p(z,\bar z)$ be a polynomial in $\mathcal{H}^n$. By Proposition \ref{almansi} we can write 
\begin{equation*}
p(z,\bar{z})=q_0(z,\bar{z})+(1-|z|^2)q_1(z,\bar{z})+\ldots+(1-|z|^2)^{n-1}q_{n-1}(z,\bar{z}),
\end{equation*}
for some $q_0(z,\bar z),q_1(z,\bar z),\ldots,q_{n-1}(z, \bar z) \in \mathcal{H}^1$. An application of Lemma \ref{decompgenharmpolinteger} gives the inclusion $(1-|z|^2)^j q_j(z,\bar{z}) \in \mathcal{E}_{0,j} \cap \bar{\mathcal{E}}_{0,j} \oplus \mathcal{E}_{1,j} \oplus \bar{\mathcal{E}}_{1,j}$ for $0 \leq j \leq n-1$, from which the result follows. 
\end{proof}

It can be worth to note that 
\begin{equation*}
\bigoplus_{j=0}^{n} \mathcal{E}_{0,j} \cap \bar{\mathcal{E}}_{0,j}=\textnormal{span}_K\{(1-|z|^2)^j: 0 \leq j \leq n\}=: P_n,
\end{equation*}
and so is the subspace $P_n \subset K[|z|^2]$ of all polynomials in $K[|z|^2]$ of degree less than or equal to $n \in \N$.

By the comments that were made earlier, we can write $\mathcal{H}=\cup_n\mathcal{H}^n$, where we recall that $\mathcal{H}^0 := \{0\}$ and that $\mathcal{H}^1 \subset \mathcal{H}^2 \subset \ldots \subset K[z,\bar z]$ are the spaces of poly\-harmonic polynomials of order $n \in \N_+$. The following result reinstates that which was established in section three, namely, that all polynomials in $K[z,\bar z]$ can be expressed in terms of the generalized harmonic polynomials $e_{m,n}(z,\bar z)$ in \eqref{bases}.
\begin{prop}
The algebra $K[z,\bar z]$ is the $K$-vector direct sum of the polynomial ring $K[|z|^2]$ together with the generalized harmonic spaces $\mathcal{E}_{1,n}, \bar{\mathcal{E}}_{1,n}$ for $n \geq 0$. In symbols, 
\begin{equation*}
K[z,\bar z] =\bigoplus_{j=0}^{\infty} \mathcal{E}_{1,j} \ \bigoplus_{j=0}^{\infty}  \bar{\mathcal{E}}_{1,j} \bigoplus K[|z|^2].
\end{equation*} 
\end{prop}
\begin{proof}
In view of the comments preceding the statement, this can be taken as a consequence  of Proposition \ref{polyharmonicinclusiongenharmpol}.
\end{proof}

We shall pass to the function setting later on. The polynomial ring can then be treated as a mere dense  subset of some larger space, for example $C^{k}(A)$, the space of $k$-times differentiable functions on some compact subset $A$ of the unit disc $\D$. The gist of the last statement may then be taken as the assertion that the span of the generalized harmonic polynomials $e_{m,n}$ are dense in this environment.

Another important case concerns the spaces $\mathcal{H}_{\alpha,\alpha}$ associated with
\begin{equation}
\label{nnharmonic}
L_{\alpha,\alpha}=(1-|z|^2)\partial \bar{\partial} +\alpha z\partial+\alpha\bar z \bar \partial - \alpha^2 \in \mathfrak{H}_2,
\end{equation}  
for $\alpha \in K$. The solutions $u \in C^{2}(\D)$ to the equation $L_{\alpha,\alpha}u(z)=0$ were thoroughly treated by A.~Olofsson in \cite{O14}, and includes the case for $\alpha=n \in \N$. They have the special property that $u \in C^{2}(\D)$ is a solution to the equation $L_{\alpha,\alpha}u=0$ if and only if $L_{\alpha,\alpha}\bar{u}=0$, and so are symmetric under conjugation. The case $\alpha=n \in \N$ is also special in the sense that all the hypergeometric functions appearing in \eqref{pqpolrepresentations} are polynomial expressions in $|z|^2=z \cdot \bar z$, regardless of $m \in \N$. To be explicit,  
\begin{equation}
\label{hypergeometricsum}
F(-n,|m|-n,|m|+1;|z|^2)=\sum_{j=0}^{n}\frac{(-n)_{j}(|m|-n)_j}{(|m|+1)_{j}}\frac{|z|^{2j}}{j!},
\end{equation}
for $m \in \Z$. As before, the notation $F(a,b,c;\cdot)$ is here used as short for the hyper\-geometric function $_2F_1(a,b,c;\cdot)$, and the symbols $(a)_n=a(a+1)\ldots(a+n-1)$ that appear above are the rising factorials or Pochhammer symbols. In general, the sum of the coefficients of such a hypergeometric function $F(a,b,c;\cdot)$ satisfies the identity
\begin{equation}
\label{vandermonde}
F(a,b,c;1)=\frac{(c-b)_{-a}}{(c)_{-a}},
\end{equation}
when $a$ is a non-positive integer. This is the so called Vandermonde identity, also known as the Chu-Vandermonde identity, and follows from Gauss more general theorem \cite[Corollary 2.2.3]{AAR}. 

We will proceed with the representational forms for $K[z,\bar z]$ in terms of the spaces $\mathcal{H}_{n,n}$ for $n \in \N$. To this end, we shall give the following important lemma.

\begin{lemma}
\label{Cdecomplemma}
Let $m,n \in \N$ and set
\begin{equation}
\label{hyperfunctionCdecomplemma}
\mathit{o}^{m,n}_l(|z|^2)=F(-n+l,m-n+l,m+1;|z|^2), \quad 0 \leq l \leq n.
\end{equation}
Then the polynomials in $K[|z|^2]$ given by  
\begin{equation}
\label{ApolynomialsCdecomplemma}
\mathit{O}_{l}^{m,n}(|z|^2)=(1-|z|^2)^l\mathit{o}^{m,n}_l(|z|^2), \quad 0 \leq l \leq n,
\end{equation}
are linearly dependent over $K$. 
\end{lemma}
\begin{proof}
Let $m,n \in \N$ and write $\mathit{o}^{m,n}_l \in K^{\star}$ in the form of
\begin{equation}
\label{Cdecomplemmahyper}
\mathit{o}^{m,n}_l(|z|^2)=\sum_{j=0}^{n-l}\vartheta_{l,j}^{m,n}|z|^{2j}, \quad 0 \leq l \leq n,
\end{equation}
where
\begin{equation}
\label{Cdecomplemmacoefficients}
\vartheta_{l,j}^{m,n}=\frac{(-n+l)_j(m-n+l)_j}{(m+1)_j j!}, \quad 0 \leq j \leq n-l,
\end{equation} 
for $0 \leq l \leq n$. 

Suppose that
\begin{equation}
\label{linindcondCdecomplemma}
0=c_0\mathit{O}_{0}^{m,n}(|z|^2)+c_1\mathit{O}_{1}^{m,n}(|z|^2)+\ldots+c_{n}\mathit{O}_{n}^{m,n}(|z|^2),
\end{equation}
for some $c_j \in K$. We then see from \eqref{ApolynomialsCdecomplemma} and \eqref{Cdecomplemmahyper} that the first term of this last expression can be written in the form of
\begin{equation*}
\mathit{O}_{0}^{m,n}(|z|^2)=\vartheta_{0,0}^{m,n}+\vartheta_{0,1}^{m,n}|z|^2+\ldots+\vartheta_{0,n}^{m,n}|z|^{2n}.
\end{equation*}
Note that 
\begin{equation}
\label{identitybasisCdecomplemma}
|z|^{2k}=\sum_{i=0}^{k} \binom{k}{i} (-1)^i(1-|z|^2)^i=1+R_k(1-|z|^2), \quad k \geq 0,
\end{equation}
for some polynomial $R_k(x) \in K[x]$ such that $R_k(0)=0$. Note further that each term $\mathit{O}_{l}^{m,n}(|z|^2)=(1-|z|^2)^l\mathit{o}^{m,n}_l(|z|^2)$ to the right of the equality in \eqref{linindcondCdecomplemma} apart from the first contains a factor $(1-|z|^2)^l$ to an order no less than $l > 0$. This allows us to express \eqref{linindcondCdecomplemma} in the form of
\begin{equation*}
0=c_0(\vartheta_{0,0}^{m,n}+\vartheta_{0,1}^{m,n}+\ldots+\vartheta_{0,n}^{m,n})+R(1-|z|^2),
\end{equation*}
for some polynomial $R(x) \in K[x]$ such that $R(0)=0$. Unless $c_0 = 0$, linear independence entails that
\begin{equation*}
\vartheta_{0,0}^{m,n}+\vartheta_{0,1}^{m,n}+\ldots+\vartheta_{0,n}^{m,n}=0.
\end{equation*}  
But the left hand sum of this last expression is given by the Chu-Vandermonde identity in \eqref{vandermonde}. In comparison with \eqref{vandermonde} and \eqref{hyperfunctionCdecomplemma}, we see that 
\begin{equation*}
\vartheta_{0,0}^{m,n}+\vartheta_{0,1}^{m,n}+\ldots+\vartheta_{0,n}^{m,n}=\frac{(n+1)_{n}}{(m+1)_{n}}.
\end{equation*}
The right hand side of this expression is different from zero however, and so $c_0=0$.  

We proceed inductively. Suppose that 
\begin{equation*}
c_0=c_1=\ldots=c_{l-1}=0. 
\end{equation*}
The expression in \eqref{linindcondCdecomplemma} can then be written as
\begin{equation*}
0=c_l\mathit{O}_{l}^{m,n}(|z|^2)+c_{l+1}\mathit{O}_{l+1}^{m,n}(|z|^2)+\ldots+c_{n}\mathit{O}_{n}^{m,n}(|z|^2).
\end{equation*}
If we now expand the first term and employ the identity that was given in \eqref{identitybasisCdecomplemma} in a similar way to that of the previous paragraph, then
\begin{align*}
0=c_l\big(\vartheta_{l,0}^{m,n}+\vartheta_{l,1}^{m,n}+\ldots+\vartheta_{l,n-l}^{m,n}\big).
\end{align*} 
Another application of the Vandermonde identity gives that  
\begin{equation*}
\vartheta_{l,0}^{m,n}+\vartheta_{l,1}^{m,n}+\ldots+\vartheta_{l,n-l}^{m,n}=\frac{(n+1-l)_{n-l}}{(m+1)_{n-l}} \neq 0.
\end{equation*}
Hence $c_l=0$, and we may conclude by induction. 
\end{proof}
\begin{prop}
\label{Olofssonbasis}
Let $n \in \N$. Then the polynomials 
\begin{equation*}
\mathit{O}_{0}^{m,n}(|z|^2),\mathit{O}_{1}^{m,n}(|z|^2),\ldots,\mathit{O}_{n}^{m,n}(|z|^2) \in K[|z|^2],
\end{equation*}
form a $K$-basis for the $(n+1)$-dimensional subspace of all polynomials in $K[|z|^2]$ of degree at most $n$ for all $m \in \N$.  
\end{prop}
\begin{proof}
Note that each of the polynomials $\mathit{O}_{0}^{m,n}(x),\ldots,\mathit{O}_{n}^{m,n}(x)$ are of degree $n$. Since there are $n+1$ such linearly independent polynomials and since $1,x,x^2,\ldots,x^{n}$ form a $K$-basis for the $(n+1)$-dimensional subspace of polynomials in $K[x] \cong K[|z|^2]$ of degree at most $n$, we may now conclude.    
\end{proof}

The significance of these last two results rests in the possibility of altering bases. While the current restriction to finite sums is comparatively forgiving, we will need to have some control over the transformation coefficients that result from such transformations later on. This involves a few technical details or arguments that can be dealt with either now or later. Since they are very similar in nature to those that were given in the proof of Lemma \ref{Cdecomplemma}, we have chosen to include them at this point rather than later, and leave any closer inspection at choice for now. We have also chosen to be rather direct in this approach, and to not rely on too many antecedents. 

\begin{lemma}
\label{basistransformlemmazerocoeff}
Let $l,m,n \in \N$ be such that $0 \leq l \leq n$ and write
\begin{equation}
\label{basistransformcoefftechlemma}
(1-|z|^2)^l=t_{l,0}^{m,n}\mathit{O}_{0}^{m,n}(|z|^2)+t_{l,1}^{m,n}\mathit{O}_{1}^{m,n}(|z|^2)+\ldots+t_{l,n}^{m,n}\mathit{O}_{n}^{m,n}(|z|^2),
\end{equation}
for some $t_{l,j}^{m,n} \in K$. Then $t_{l,j}^{m,n}=0$ for $0 \leq j < l$. 
\end{lemma}
\begin{proof}
We shall consult the proof procedure that was given in Lemma \ref{Cdecomplemma}. 

Recall the quotient of Pochhammer symbols $\vartheta_{k,j}^{m,n}$ in \eqref{Cdecomplemmacoefficients} of Lemma \ref{Cdecomplemma},
\begin{equation}
\label{Ccoefficientslemma}
\vartheta_{k,j}^{m,n}=\frac{(-n+k)_j(m-n+k)_j}{(m+1)_j j!}, \quad 0 \leq j \leq n-k,
\end{equation}
for $0 \leq k \leq n$.

They were the coefficients of the hypergeometric function $\mathit{o}_k^{m,n}\in K^{\star}$ in \eqref{hyperfunctionCdecomplemma}, that expresses 
\begin{equation}
\label{bigOsmallotransfcoefflemma}
\mathit{O}_{k}^{m,n}(|z|^2)=(1-|z|^2)^k\mathit{o}_k^{m,n}(|z|^2), \quad 0 \leq k \leq n. 
\end{equation}
In such terms,
\begin{equation}
\label{ObasisinCcoeff}
O_k^{m,n}(|z|^2)=(1-|z|^2)^k\big(\vartheta_{k,0}^{m,n}+\vartheta_{k,1}^{m,n}|z|^2+\ldots+\vartheta_{k,n-k}^{m,n}|z|^{2(n-k)}\big),
\end{equation}
for $0 \leq k \leq n$. Via a similar transformation to that in \eqref{identitybasisCdecomplemma}, we can then write
\begin{equation}
\label{ObasisinCcoeffsimplified}
O_k^{m,n}(|z|^2)=(1-|z|^2)^k\big(\vartheta_{k,0}^{m,n}+\vartheta_{k,1}^{m,n}+\ldots+\vartheta_{k,n-k}^{m,n}\big)+R(1-|z|^2),
\end{equation}
for some $R(x) \in K[x]$ whose smallest degree term exceeds $k$. Linear independence and the identity of Chu-Vandermonde then entails that 
\begin{equation*}
0=t_{l,0}^{m,n}(\vartheta_{0,0}^{m,n}+\vartheta_{0,1}^{m,n}+\vartheta_{0,2}^{m,n}+\ldots+\vartheta_{0,n}^{m,n})=t_{l,0}^{m,n}\frac{(n+1)_{n}}{(m+1)_{n}}, \quad l > 0.
\end{equation*} 
This shows that $t_{l,0}^{m,n}=0$ whenever $l > 0$. If $t_{l,0}^{m,n}=\ldots=t_{l,k-1}^{m,n}=0$ for $0 \leq k < l$, this same expression \eqref{ObasisinCcoeffsimplified} implies that 
\begin{equation*}
0=t_{l,k}^{m,n}(\vartheta_{k,0}^{m,n}+\vartheta_{k,1}^{m,n}+\ldots+\vartheta_{k,n-k}^{m,n})(1-|z|^2)^k, \quad 0 \leq k < l. 
\end{equation*}
Another application of the Vandermonde identity then shows that $t_{l,k}^{m,n}=0$ for $0 \leq k < l$. This completes the proof by induction. 
\end{proof}

\begin{lemma}
\label{basistransformlemmanonzerocoeff}
Let $l,m,n \in \N$ be such that $0 \leq l \leq n$ and let $t_{l,k}^{m,n} \in K$ be as in \eqref{basistransformcoefftechlemma} for $0 \leq k \leq n$. Set
\begin{equation}
\label{sumsrationalcoefflemma}
S_{j,k}^{m,n}=(-1)^{k-j} \sum_{i=0}^{n-k} \binom{k-j+i}{k-j}\vartheta_{j,k-j+i}^{m,n}, \quad 0 \leq j \leq k,
\end{equation}
where 
\begin{equation}
\label{Ccoefficientslemma}
\vartheta_{k,j}^{m,n}=\frac{(-n+k)_j(m-n+k)_j}{(m+1)_j j!}, \quad 0 \leq j \leq n-k,
\end{equation}
for $0 \leq k \leq n$. Then 
\begin{equation}
\label{rationalcoeffnonzerolemma}
t_{l,k}^{m,n}=-\frac{(m+1)_{n-k}}{(n+1-k)_{n-k}}\Lambda_{l,k}^{m,n},
\end{equation}
where $\Lambda_{l,k}^{m,n}=-1$ for $k=l$ and $\Lambda_{l,k}^{m,n}=t_{l,k-1}^{m,n}S_{k-1,k}^{m,n}+\ldots+t_{l,l}^{m,n}S_{l,k}^{m,n}$ for $l < k \leq n$.
\end{lemma}
\begin{proof}
By the preceding Lemma \ref{basistransformlemmazerocoeff}, we have that $t_{l,j}^{m,n}=0$ for $0 \leq j <l$. It thus remains to compute $t_{l,j}^{m,n}$ for $l \leq j \leq n$. Note from \eqref{basistransformcoefftechlemma} that all the remaining terms except the $l$:th term carry a factor $(1-|z|^2)^j$ of degree $j > l$. Hence, 
\begin{align*}
(1-|z|^2)^l &=t_{l,l}^{m,n}(\vartheta_{l,0}^{m,n}+\vartheta_{l,1}^{m,n}+\ldots+\vartheta_{l,n-l}^{m,n})(1-|z|^2)^l.
\end{align*} 
Another application of the Vandermonde identity then shows that 
\begin{equation*}
t_{l,l}^{m,n}=\frac{(m+1)_{n-l}}{(n+1-l)_{n-l}}.
\end{equation*}
This gives the first part of the statement. As for $k > l$, note that the lowest order term to which $t_{l,k}^{m,n}$ appears in \eqref{basistransformcoefftechlemma} following the transformation in \eqref{identitybasisCdecomplemma} is
\begin{align}
\label{computationcoeffsurvivingterms}
t_{l,k}^{m,n}(\vartheta_{k,0}^{m,n}+\ldots+\vartheta_{k,n-k}^{m,n})(1-|z|^2)^k=t_{l,k}^{m,n}\frac{(n+1-k)_{n-k}}{(m+1)_{n-k}}(1-|z|^2)^k.
\end{align} 
Since $k > l$, this term must equate to zero together with the rest of terms in \eqref{basistransformcoefftechlemma} for which there appears a factor of $(1-|z|^2)^k$. By above, the only contributing terms in \eqref{basistransformcoefftechlemma} are those terms $O_j^{m,n}$ for which $l \leq j \leq k$. This since $t_{l,j}^{m,n}=0$ for $j < l$ by the preceding lemma, and no term $O_j^{m,n}$ for which $j > k$ carry a factor of $(1-|z|^2)^j$ of degree less than or equal to $k$. We shall express the hypergeometric functions $\mathit{o}_j^{m,n} \in K^{\star}$ \eqref{hyperfunctionCdecomplemma} in the form of
\begin{equation}
\label{hypercoefflemma}
F(-n+j,m-n+j,m+1;|z|^2)=\sum_{p=0}^{n-j} \sum_{q=0}^{p} \vartheta_{j,p}^{m,n}\binom{p}{q}(-1)^q(1-|z|^2)^q,
\end{equation}
for $0 \leq j \leq n$. If we now multiply this last expression by a factor of $(1-|z|^2)^j$, we see that the coefficient of the $k$:th degree term in $1-|z|^2$ that appears in the resulting expression is 
\begin{equation*}
S_{j,k}^{m,n}:=(-1)^{k-j} \sum_{i=0}^{n-k} \binom{k-j+i}{k-j}\vartheta_{j,k-j+i}^{m,n},
\end{equation*}
for $0 \leq j \leq k$. Since we have already established that $t_{l,j}^{m,n}=0$ for $0 \leq j < l$, we can restrict to $l \leq j \leq k$. We also see from this expression that there are no terms of degree less than or equal to $k$ when $j > k$, as already noticed. If we now recall \eqref{computationcoeffsurvivingterms} and equate the two sides in \eqref{basistransformcoefftechlemma}, we see that,
\begin{equation*}
0=\bigg(t_{l,k}^{m,n}\frac{(n+1-k)_{n-k}}{(m+1)_{n-k}}+t_{l,k-1}^{m,n}S_{k-1,k}^m+\ldots+t_{l,l}^{m,n}S_{l,k}^{m,n}\bigg)(1-|z|^2)^k,
\end{equation*}
for $l< k \leq n$. The desired result then follows by a rearrangement of terms.
\end{proof}

\begin{prop}
\label{rationalfunctionscoeff}
Let $l,m,n \in \N$ be such that $0 \leq l \leq n$ and let $t_{l,k}^{m,n} \in K$ be as in \eqref{basistransformcoefftechlemma} for $0 \leq k \leq n$. Then there is a rational expression $t_{l,k}^n(x) \in S_{\N}^{-1}K[x]$ such that $t_{l,k}^n(m)=t_{l,k}^{m,n}$ for $m \geq 0$ and $n \geq k \geq 0$. 
\end{prop}
\begin{proof}
We saw in Lemma \ref{basistransformlemmazerocoeff} that $t_{l,k}^{m,n}=0$ for $0 \leq k < l$. 

As for $l \leq k \leq n$, note from \eqref{sumsrationalcoefflemma} and \eqref{Ccoefficientslemma} that $S_{j,k}^{m,n}$ is a well defined expression for all $m \geq 0$. We also recall that $S_{\N}^{-1}K[x]$ is the localization of $K[x]$ by $S_{\N}$ and set
\begin{equation*}
S_{j,k}^n(x)=(-1)^{k-j} \sum_{i=0}^{n-k} \binom{k-j+i}{k-j}\vartheta_{j,k-j+i}^{n}(x) \in S_{\N}^{-1}K[x],
\end{equation*}
for $0 \leq j \leq k$, where 
\begin{equation*}
\vartheta_{k,j}^{n}(x)=\frac{(-n+k)_j(x-n+k)_j}{(x+1)_j j!} \in S_{\N}^{-1}K[x], 
\end{equation*}
for $0 \leq j \leq n-k$. Let 
\begin{equation*}
t_{l,l}^n(x)=\frac{(x+1)_{n-l}}{(n+1-l)_{n-l}} \in S_{\N}^{-1}K[x],
\end{equation*}
and recursively define
\begin{equation*}
t_{l,k}^n(x)=-\frac{(x+1)_{n-k}}{(n+1-k)_{n-k}}\Lambda_{l,k}^n(x) \in S_{\N}^{-1}K[x],
\end{equation*}
where 
\begin{equation*}
\Lambda_{l,k}^n(x)=t_{l,k-1}^n(x)S_{k-1,k}^m(x)+\ldots+t_{l,l}^{n}(x)S_{l,k}^n(x) \in S_{\N}^{-1}K[x],
\end{equation*}
for $l < k \leq n$. It is then evident from \eqref{rationalcoeffnonzerolemma} that $t_{l,k}^{m,n}$ is the image of $\N$ under $t_{l,k}^n(x)$ for $l \leq k \leq n$, such that $t_{l,k}^n(m)=t_{l,k}^{m,n}$ for $m \geq 0$. 
\end{proof}

\begin{cor}
Let $l,m,n \in \N$ be such that $0 \leq l \leq n$ and let $t_{l,k}^{m,n} \in K$ be as in \eqref{basistransformcoefftechlemma} for $0 \leq k \leq n$. Then $t_{l,k}^{m,n} \in \Q$ is rational for $0 \leq k \leq n$.
\end{cor}
\begin{proof}
This is evident from the previous Proposition \ref{rationalfunctionscoeff}. 
\end{proof}

By Proposition \ref{linearindependencehomparts}, each $p(z,\bar{z}) \in K[z,\bar z]$ may be written in the form 
\begin{equation}
\label{mthcomponent}
p(z,\bar{z})=\sum_{m=-n}^{n} p_m \xi_m(z),
\end{equation}
for some polynomials $p_{-n},\ldots,p_{n}$ in $K^{\star}$, where we recall the use of the symbol $\xi$ to denote powers in $z$ and $\bar z$ from \eqref{xisymbol}. Let $\mathcal{S}$ be the collection of integers $m \in \Z$ in the interval $-n \leq m \leq n$, and recall the projection operators $\pi_m=\pi_m^{\mathcal{S}}$ from \eqref{projectionmaps}. For any integer $m \in\mathcal{S}$, we shall then refer to the $m$:th component of $p(z,\bar{z})$ as that polynomial $\pi_mp(z,\bar{z})$ obtained from $p(z,\bar{z})$ by projecting it onto the submodule $K^{\star}\xi_m(z)$. In the example just given \eqref{mthcomponent}, this polynomial is precisely $p_m\xi_m(z)$. We will sometimes also refer to the $m$:th component of the polynomial $p(z,\bar z)$ without referring to any explicit construction or operational procedure.    

\begin{thm}
\label{Cdecompthm}
Let $n \in \N$ be positive and let $m \in \Z$. Suppose that $p(z,\bar{z}) \in \mathcal{H}^n$ is a poly\-harmonic polynomial in the $K^{\star}$-module $\mathcal{H}$ of order $n$ and let $q_m(z,\bar{z})$ be its projection onto the $m$:th component. Then 
\begin{equation}
\label{Cellulardecompositionstatement}
q_m(z,\bar{z})=k_{m,0}\mathit{O}_{0}^{|m|,n-1}\xi_m(z)+\ldots+k_{m,n-1}\mathit{O}_{n-1}^{|m|,n-1}\xi_m(z),
\end{equation}
for some $k_{m,j} \in K$, where $O_j^{m,n-1} \in K^{\star}$ are the polynomials that were introduced in \eqref{ApolynomialsCdecomplemma}.  
\end{thm}
\begin{proof}
Let $p(z,\bar{z})$ be a polyharmonic polynomial in $\mathcal{H}^n$ of order $n$, and recall the Almansi representation from Proposition \ref{almansi} that expresses this poly\-nomial in the form of
\begin{equation}
\label{sumCdecomthm}
p(z,\bar{z})=p_0(z,\bar{z})+(1-|z|^2)p_1(z,\bar{z})+\ldots+(1-|z|^2)^{n-1}p_{n-1}(z,\bar{z}),
\end{equation}
for some harmonic polynomials $p_0(z,\bar{z}),\ldots,p_{n-1}(z,\bar{z}) \in \mathcal{H}^1$. The polynomials $p_j(z,\bar z) \in \mathcal{H}^1$ can be written in the form 
\begin{equation*}
p_j(z,\bar{z})=\sum_{i=0}^{N_j} c_{j,i} z^i + \sum_{i=1}^{N_j}c_{j,-i}\bar{z}^i, \quad 0 \leq j \leq n-1,
\end{equation*} 
for some $c_{j,i} \in K$, and are of degree at most $N_j \in \N$. By further extending such sums by zero if necessary, we may as well assume that they are all of the same maximum degree $M=\max\{{N_0,N_1,\ldots,N_{n-1}}\}$. We can then write \eqref{sumCdecomthm} in the form of 
\begin{equation*}
p(z,\bar{z})=\sum_{j=0}^{n-1}\sum_{i=0}^{M} c_{j,i}(1-|z|^2)^j z^i + \sum_{j=0}^{n-1}\sum_{i=1}^{M}c_{j,-i}(1-|z|^2)^j\bar{z}^i.
\end{equation*}  
Let $m \in \Z$ be such that $-M \leq m \leq M$ and set
\begin{equation*}
q_m(z,\bar{z})=\pi_m p(z,\bar{z}),
\end{equation*}  
for the projection of $p(z,\bar{z})$ onto its $m$:th component. Then
\begin{equation}
\label{mcomponentsumthm}
q_m(z,\bar{z})=\xi_m(z)\sum_{j=0}^{n-1} c_{j,m}(1-|z|^2)^j.
\end{equation}
Notice that each term in the sum is a polynomial in $K[|z|^2]$ of degree no greater than $n-1$. We now know from Proposition \ref{Olofssonbasis} that the $n$ poly\-nomials  
\begin{equation*}
\mathit{O}_{0}^{|m|,n-1}(|z|^2),\ldots,\mathit{O}_{n-1}^{|m|,n-1}(|z|^2) \in K[|z|^2],
\end{equation*}
make a $K$-basis for all polynomials in $K[|z|^2]$ of degree at most $n-1$. In a change of basis,
\begin{equation*}
(1-|z|^2)^l=\sum_{j=0}^{n-1} t_{l,j}^{|m|,n-1} \mathit{O}_{j}^{|m|,n-1}(|z|^2),
\end{equation*}   
for some rational coefficients $t_{l,j}^{|m|,n-1} \in \Q \subset K$, where $0 \leq l \leq n-1$. Inserting the last expression back into the sum in \eqref{mcomponentsumthm} gives 
\begin{equation*}
q_m(z,\bar{z})=\xi_m(z)\sum_{j=0}^{n-1} \sum_{i=0}^{n-1} c_{j,m} t_{j,i}^{|m|,n-1} \mathit{O}_{i}^{|m|,n-1}(|z|^2).
\end{equation*}
We obtain \eqref{Cellulardecompositionstatement} by letting
\begin{equation*}
k_{m,i}=\sum_{j=0}^{n-1} c_{j,m}t_{j,i}^{|m|,n-1}, \quad 0 \leq i \leq n-1.
\end{equation*}
This completes the proof of this theorem. 
\end{proof}
  
As to the implications that follow, we also recall from earlier that each polynomial in $K[z,\bar z]=\bigcup_n \mathcal{H}^n$ is polyharmonic to some order $n \in \N$. 
\begin{cor}
\label{borhedsums}
Let $n \in \N$ be positive and let $p(z,\bar{z}) \in \mathcal{H}^n$ be a polyharmonic polynomial in the $K^{\star}$-module $\mathcal{H}$ of order $n$. Then
\begin{align}
\label{sumsofborhed}
p(z,\bar{z})&=\sum_{m} k_{m,0}\mathit{O}_{0}^{|m|,n-1}\xi_m(z)+\ldots+\sum_{m} k_{m,n-1}\mathit{O}_{n-1}^{|m|,n-1} \xi_m(z),
\end{align}
for some $k_{m,j} \in K$ that are non-zero for at most finitely many indices $m \in \Z$ and $0 \leq j \leq n-1$.
\begin{proof}
The polynomial $p(z,\bar z) \in K[z,\bar z]$ is the sum of its components, and by Theorem \ref{Cdecompthm}, each such component is given by 
\begin{equation*}
\pi_m p(z,\bar z)=k_{m,0}\mathit{O}_{0}^{|m|,n-1}\xi_m(z)+\ldots+k_{m,n-1}\mathit{O}_{n-1}^{|m|,n-1}\xi_m(z), \quad m \in \Z,
\end{equation*} 
for some $k_{m,j} \in K$. Since $p(z,\bar z)$ is a polynomial, the number of such components is finite. 
\end{proof}
\end{cor}
Theorem \ref{Cdecompthm} leads to the so called cellular decomposition for polynomials in $K[z,\bar z]$, which has acquired its name from the representation for polyharmonic functions that was first provided by A.~Borichev and H.~Hedenmalm in \cite{BH}. We follow practice and introduce the endomorphisms $M^l$ on $K[z,\bar z]$ that act via multi\-plication by the basis elements of $K[|z|^2]$ in the form of   
\begin{equation}
\label{Moperator}
M^lq(z,\bar z)=(1-|z|^2)^lq(z,\bar z), \quad l \geq 0.
\end{equation}
We also recall the representation for the $(\gamma_1,\gamma_2)$-harmonic polynomials that were given in \eqref{pqpolrepresentations} and the hypergeometric polynomials $\mathit{o}_l^{m,n} \in K^{\star}$ from \eqref{hyperfunctionCdecomplemma}, and write
\begin{equation}
\label{olofssonsumsborhedpol}
w_l(z,\bar z)=\sum_{m=-\infty}^{\infty} k_{m,l}\mathit{o}_{l}^{|m|,n-1}\xi_m(z), \quad 0 \leq l\leq n-1,
\end{equation}
for some sequence of elements $k_{m,l} \in K$ that are non-zero for at most finitely many $m \in \Z$. By the comments that were made in connection with \eqref{pqpolrepresentations}, each such polynomial $w_l(z,\bar z)$ is a solution to the equation 
\begin{equation*}
L_{n-l-1,n-l-1}w_l(z,\bar z)=0,
\end{equation*}
for $0 \leq l \leq n-1$. The representation for polyharmonic polynomials that follows will later be given in a more general setting. 
\begin{cor}
\label{cellulardecompthmpolynomials}
Let $n \in \N$ be positive and let $p(z,\bar{z}) \in \mathcal{H}^n$ be a polyharmonic polynomial in $K[z,\bar z]$ of order $n$. Then the polynomial $p(z,\bar z)$ can be written in the form of
\begin{align}
\label{cellulardecomp}
p(z,\bar{z})&=w_0(z,\bar z)+M^1w_1(z,\bar z)+\ldots+M^{n-1}w_{n-1}(z,\bar z),
\end{align}
where $M^l$ is the operator in \eqref{Moperator}, and with $w_l$ given in the form of \eqref{olofssonsumsborhedpol}, relative to some finite sequence $\{k_{m,l} \}_{m \in \Z}$ of elements in $K$ for $0 \leq l \leq n-1$. 
\end{cor}
\begin{proof}
Recall from \eqref{ApolynomialsCdecomplemma} that 
\begin{equation*}
\mathit{O}_l^{m,n-1}(|z|^2)=(1-|z|^2)^l\mathit{o}_l^{m,n-1}(|z|^2)=M^l\mathit{o}_l^{m,n-1}(|z|^2), \quad 0 \leq l \leq n-1. 
\end{equation*}
By Corollary \ref{borhedsums} and equation \eqref{sumsofborhed}, we can then write
\begin{align*}
p(z,\bar{z})&=\sum_{j=0}^{n-1} \sum_{m} k_{m,j}\mathit{O}_j^{|m|,n-1}\xi_m(z)= \sum_{j=0}^{n-1} M^j \bigg(\sum_{m}k_{m,j}\mathit{o}_j^{|m|,n-1}\xi_m(z)\bigg),
\end{align*}
for some $k_{m,l} \in K$ that are non-zero for at most finitely many indices $m \in \Z$ and $0 \leq l \leq n-1$. The conclusion is now evident from \eqref{olofssonsumsborhedpol}.
\end{proof}
In connection with the earlier discussions, we will also provide the following form for this representation, where we recall the polynomials $e_{m,n}(z,\bar z)$ that were defined in \eqref{bases}. 
\begin{cor}
Let $n \in \N$ be positive and let $p(z,\bar{z}) \in \mathcal{H}^n$ be a polyharmonic polynomial in the $K^{\star}$-module $\mathcal{H}$ of order $n$. Then
\begin{align*}
p(z,\bar{z})&=\sum_{j=0}^{n-1}\sum_{m} k_{m,j}\mathit{o}_{j}^{|m|,n-1}e_{m,j}(z,\bar z)
\end{align*}
for some $k_{m,l} \in K$ that are non-zero for at most finitely many indices $m \in \Z$ and $0 \leq l \leq n-1$.
\begin{proof}
This is just another way to express the cellular decomposition in \eqref{cellulardecomp}, noting that $M^j\xi_m(z)=e_{m,j}(z,\bar z)$ for $0 \leq j \leq n-1$. 
\end{proof}
\end{cor}

\section{Generalized harmonic functions}

We shall pass to the function setting. 

\bigskip

Let $K$ be the field of complex numbers $\C \cong \R^2$. Let $z=x+iy$ be the usual decomposition of a complex number into its real- and imaginary part, and denote its complex conjugate by $\bar z=x-iy$. The map on $K[z,\bar z]$ that takes the polynomial $p(z,\bar z)$ to the polynomial function $z \mapsto p(z,\bar z)$ for $z \in \C$ is an isomorphism of rings from $K[z,\bar z]$ to the ring of polynomial functions on $\C$ with its usual rules of addition and multiplication.\footnote{The polynomial functions may be taken as the complexification of real analytic functions, and justifies the use of the notation $p(z,\bar z)$ for such functions.} The complex differential operators or Wirtinger derivatives are given in the common sense by $\partial=(\partial_x-i\partial_y)/2$ and $\bar \partial=(\partial_x+i\partial_y)/2$, and they satisfy $\partial z=\bar \partial \bar z=1$ and $\partial \bar z = \bar \partial z=0$ on $\C$. Against this background, we may then transfer the results that were obtained for the polynomial ring $K[z,\bar z]$ to the algebra of polynomial functions on $\C$. The latter should come to be regarded as embedded within the space $C^{\infty}(\D)$ of smooth functions on $\D \subset \C$, and the complex conjugate $\bar{f}$ of a function $f \in C^{\infty}(\D)$ will as usual be given by $\bar{f}(z)=\overline{f(z)}$. The algebra $A_2(\C)$ is the algebra generated by $z,\bar z, \partial$ and $\bar \partial$ in $\textnormal{End}(C^{\infty}(\D))$, and $D \in A_2(\C)$ is said to commute with rotations if $(Du) \circ R_{e^{i\theta}}=D(u \circ R_{e^{i\theta}})$ for $u \in C^{\infty}(\D)$ and $R_{e^{i\theta}} \in \textnormal{SO}(2) \cong \T$. The subalgebra $\mathfrak{R}_2 \subset A_2(\C)$ will then be defined as the set of all differential operators that commute with rotations, and $\mathfrak{H}_2 \subset \mathfrak{R}_2$ is the algebra generated by the differential operators $z \partial, \bar z \bar \partial$ and $\partial \bar \partial$ in $A_2(\C)$.  

\begin{dfn}
A generalized harmonic function $u$ on $\D$ is a function $u \in C^{n}(\D)$ that satisfies $Du=0$ in $\D$ for some $D \in \mathfrak{H}_2$, where $n \in \N$ is the order of $D \in A_2(\C)$.
\end{dfn} 

Recall the operator $L_{\gamma_1,\gamma_2} \in \mathfrak{H}_2$ from \eqref{pqoperator} for $\gamma_1,\gamma_2 \in \C$. A formal inspection of the representation in \eqref{pqpolrepresentations} for $(\gamma_1,\gamma_2)$-harmonic functions suggests that they arise from the simpler $e_{m,n}$ that were given in \eqref{bases}. Formally, we can express this relation as 
\begin{equation}
u \sim \sum_{m\in \Z}\sum_{n \in \N} c_{m,n} e_{m,n},
\end{equation}
following a rewriting of the hypergeometric functions in \eqref{pqpolrepresentations} and a rearrangement of terms. This gives us a hint on how to proceed next. As in the case of the finite setting, we shall narrow the numbers somewhat and restrict to the case of integer parameters. The hyper\-geometric functions appearing in \eqref{pqpolrepresentations} are then polynomial expressions in $z\bar z=|z|^2$ over $\Q$, and the restricted setting $\gamma_1,\gamma_2 \in \N$ is in fact the only for which each of the hyper\-geometric functions in \eqref{pqpolrepresentations} are reduced to polynomial expressions over $\Q$. Complex-wise, we also have that a function $u$ is $(\gamma_1,\gamma_2)$-harmonic if and only if its conjugate $\bar u$ is $(\bar{\gamma}_2,\bar{\gamma}_1)$-harmonic. Hence, the only setting for which all the hypergeometric functions are polynomial expressions in $|z|^2$ on $\D$ while both $u$ and its conjugate $\bar u$ are $(\gamma_1,\gamma_2)$-harmonic, is the case where $\gamma_1=\gamma_2=n$ for some $n \in \N$. These remarks should be taken from the view point of the earlier discussion, and in particular, that which followed $\eqref{pqoperator}$ and \eqref{nnharmonic}. 

We shall proceed by transferring some of the results that were given in the confinement to polynomials, and from the polynomial ring $K[z,\bar z]$ to the ring $C^{\infty}(\D)$ of smooth functions on $\D$. Our main concern will be that of Proposition \ref{pqtopolyprop} and the cellular decomposition in Corollary \ref{cellulardecompthmpolynomials}. The first of these two concerned the way in which $(\gamma_1,\gamma_2)$-harmonic polynomials can be represented in terms of polyharmonic polynomials, while the latter dealt with the inverse relationship. We will start with the former, and regard the subsequent proceedings as an instructive way to describe how such functions may arise. As for the finite setting, the conclusion itself is to expected from the following preliminary discussion.       

A polyharmonic function of order $n$ on the unit disc $\D$ is a complex function $u$ that is continuously differentiable to some appropriate degree and satisfies 
\begin{equation*}
\partial^n\bar \partial^n u=0, \quad \textnormal{in} \ \D. 
\end{equation*}
As in the case of polynomials, it can be shown that a function $u \in C^{2n}(\D)$ is a polyharmonic function of order $n$ on $\D$ if and only if 
\begin{equation}
\label{almansiexpfunctions}
u(z)=v_0(z)+(1-|z|^2)v_1(z)+\ldots+(1-|z|^2)^{n-1}v_{n-1}(z), \quad z \in \D,
\end{equation}
for some harmonic functions $v_i \in C^{\infty}(\D)$ such that $\partial \bar \partial v_i=0$ holds in $\D$ for every $0 \leq i \leq n-1$ \cite[Section 32]{Gakhov}. From \eqref{pqpolrepresentations} we also know that each of the latter can be written in the form 
\begin{equation}
\label{harmonicsumfunctions}
v_i(z)=\sum_{m=-\infty}^{\infty}c_{i,m}\xi_m(z), \quad 0 \leq i \leq n-1,
\end{equation}
for $z \in \D$ and some $c_{i,m} \in \C$ that fulfil \eqref{limsupcoeff} for $0 \leq i \leq n-1$. 
\begin{prop}
\label{pqtopolpropinfinite}
Let $\gamma_1,\gamma_2 \in \N$. Let $u \in C^{\infty}(\D)$ be a $(\gamma_1,\gamma_2)$-harmonic function on $\D$. Then $u$ can be written as
\begin{equation}
u(z)=P_1(z)+P_2(z)+P_3(z), \quad z \in \D, 
\end{equation}
for some functions $P_i \in C^{\infty}(\D)$, where $P_1$ is polyharmonic to an order at least $\gamma_1+1 \in \N$, $P_2$ to an order at least $\gamma_2+1 \in \N$ and the function $P_3$ to an order at least $\min(\gamma_1+1,\gamma_2+1) \in \N$. Thus, every $(\gamma_1,\gamma_2)$-harmonic function that is parametrized by the non-negative whole numbers is polyharmonic to an order at least $\max(\gamma_1+1,\gamma_2+1)\in \N$.   
\end{prop}
\begin{proof}
Let $\gamma_1,\gamma_2 \in \N$. Since $u \in C^{\infty}(\D)$ is a $(\gamma_1,\gamma_2)$-harmonic function, it can be written in the form of \eqref{pqpolrepresentations} for some complex numbers $c_m \in \C$ satisfying \eqref{limsupcoeff}. Let $P_{1,N}(z) \in C^{\infty}(\D)$ denote the $N$:th partial sum of the first sum in \eqref{pqpolrepresentations}, such that 
\begin{equation}
\label{partialsumS}
P_{1,N}(z)=\sum_{m=1}^{N} c_m F(-\gamma_1,m-\gamma_2,m+1;|z|^2)z^m, \quad z \in \D,
\end{equation}   
for $N \geq 1$. By Lemma \ref{pqtopolylemma}, we can then write
\begin{equation*}
P_{1,N}(z)=\sum_{m=1}^{N} c_m q_m(z), \quad N \geq 1,
\end{equation*} 
where we recall from \eqref{partialsumpqtopollemma} that the polynomials $q_m \in C^{\infty}(\D)$ are of the form
\begin{equation*}
q_m(z)=\big(t_0(m)+t_1(m)(1-|z|^2)+\ldots+ t_{\gamma_1}(m)(1-|z|^2)^{\gamma_1} \big)z^m, \quad m \geq 1, \quad z \in \D,
\end{equation*}
for some $t_0(x),t_1(x),\ldots,t_{\gamma_1}(x) \in S_{\N}^{-1}K[x]$, where $S_{\N}^{-1}K[x]$ is the localization of $K[x]$ by $S_{\N}$ in \eqref{localization}. Let $M^j$ be the endo\-morphisms on $C^{\infty}(\D)$ that are defined via multiplication by the basis elements of $K[|z|^2]$ in the form of 
\begin{equation*}
M^ju(z)=(1-|z|^2)^ju(z), \quad j \geq 0,
\end{equation*}
for $z \in \D$. Consider the sums
\begin{equation*}
R_{j,N}(z)=\sum_{m=1}^N c_m t_j(m) z^m, \quad 0 \leq j \leq \gamma_1,
\end{equation*} 
for $z \in \D$. Recall that $t_j(x) \in S_{\N}^{-1}K[x]$ for $0 \leq j \leq \gamma_1$ and that the sequence of elements $c_{m} \in \C$ satisfies \eqref{limsupcoeff}. Hence, 
\begin{equation*}
\limsup_{m \rightarrow \infty}|c_mt_j(m)|^{1/m} \leq 1, \quad 0 \leq j \leq \gamma_1.
\end{equation*}
Thus, and according to the comments surrounding \eqref{pqpolrepresentations}, we get that the partial sums $R_{j,N} \in C^{\infty}(\D)$ converges absolutely in the space of smooth functions on $\D$ for $0 \leq j \leq \gamma_1$. Let 
\begin{equation*}
v_{1,j}=\lim_{N \rightarrow \infty}R_{j,N},
\end{equation*}  
in $C^{\infty}(\D)$ for $0 \leq j \leq \gamma_1$. Take the limit of the sequence of partial sums in \eqref{partialsumS} and denote their limit by 
\begin{equation*}
P_1=\lim_{N \to \infty} P_{1,N}, 
\end{equation*}
in $C^{\infty}(\D)$. Then 
\begin{equation*}
P_1(z)=v_{1,0}(z)+M^1v_{1,1}(z)+\ldots+M^{\gamma_1}v_{1,\gamma_1}(z), \quad z \in \D.
\end{equation*} 
By the comments preceding this statement, $P_1 \in C^{\infty}(\D)$ is a poly\-harmonic function to an order at least $\gamma_1+1 \in \N$. 

In view of the above, the polyharmonic representation for the second sum in \eqref{pqpolrepresentations} is obtained by considering the conjugate of 
\begin{equation*}
P_{2,N}(z)=\sum_{m=1}^{N} c_{-m} F(-\gamma_2,m-\gamma_1,m+1;|z|^2)\bar{z}^m, \quad N \geq 1, \quad z \in \D.
\end{equation*}
If we denote this second sum by $P_2 = \lim_{N \to \infty} P_{2,N} \in C^{\infty}(\D)$, then 
\begin{equation*}
P_2(z)=v_{2,0}(z)+M^1v_{2,1}(z)+\ldots+M^{\gamma_2}v_{2,\gamma_2}(z), \quad z \in \D,
\end{equation*}
where each $v_{2,j} \in C^{\infty}(\D)$ is harmonic on $\D$ for $0 \leq j \leq \gamma_2$. In view of the comments that preceded this statement, we see that $P_2 \in C^{\infty}(\D)$ is polyharmonic to an order at least $\gamma_2+1 \in \N$. 

The last term in \eqref{pqpolrepresentations}, call it $P_3 \in C^{\infty}(\D)$, is easily seen to be polyharmonic to an order at least $\min(\gamma_1+1,\gamma_2+1) \in \N$. Collecting terms, we get that
\begin{equation*}
u(z)=P_1(z)+P_2(z)+P_3(z), \quad z \in \D.
\end{equation*}
Hence $u \in C^{\infty}(\D)$ is polyharmonic to an order at least $\max(\gamma_1+1,\gamma_2+1) \in \N$. 
\end{proof}
The following conclusion should be compared to Corollary 6.2 in \cite{BH}. 
\begin{cor}
\label{n-1harmonictopolycor}
Let $n \in \N$ be a positive integer and suppose that $u \in C^{\infty}(\D)$ is a $(n-1,n-1)$-harmonic function on $\D$. Then $u$ is polyharmonic to an order at least $n$ on $\D$.   
\end{cor}
\begin{proof}
This corollary follows from the more general Proposition \ref{pqtopolpropinfinite}. 
\end{proof}

We will proceed with the cellular decomposition for polyharmonic functions on $\D$, and transfer the result that was given for polynomials in Corollary \ref{cellulardecompthmpolynomials} to the more general setting of smooth functions. This representation for polyharmonic functions has been conceptualized by the two authors who first proved it and is the analogue of \eqref{cellulardecomp} in the more general setting of smooth functions on the unit disc $\D$. It is called the cellular decomposition for polyharmonic functions and was given in \cite{BH} by A.~Borichev and H.~Hedenmalm. 

We shall employ Proposition \ref{rationalfunctionscoeff} in this respect, the content of which relates to the growth of the transformation coefficients that results from the change between the various sets of bases that follow. It ensures convergence and the separation of sums. As was done for polynomials, we shall follow practice and define the endomorphisms $M^j \in \mathfrak{R}_2$ on $C^{\infty}(\D)$ that act via multiplication by the basis elements of $K[|z|^2]$ in the form of 
\begin{equation}
\label{Moperatorfunctioncase}
M^ju(z)=(1-|z|^2)^ju(z), \quad z \in \D,
\end{equation} 
for $j \geq 0$. We also recall the polynomials $\mathit{o}_{j}^{m,n} \in C^{\infty}(\D)$ in \eqref{hyperfunctionCdecomplemma} and the sums $w_j \in C^{\infty}(\D)$ from \eqref{olofssonsumsborhedpol} or \eqref{pqpolrepresentations} in the form of
\begin{equation}
\label{olofssonsumsfunctioncase}
w_j(z)=\sum_{m=-\infty}^{\infty} k_{m,j}\mathit{o}_{j}^{|m|,n-1}(|z|^2)\xi_m(z), \quad z \in \D,
\end{equation}
for some complex sequences $\{k_{m,j}\}_{m=-\infty}^{\infty}$ satisfying \eqref{limsupcoeff} for $0 \leq j\leq n-1$.

\begin{thm}
\label{cellulardecomposition}
Let $u \in C^{2n}(\D)$ be a polyharmonic function to some positive order $n \in \N$. Then $u$ can be written in the form of
\begin{equation}
\label{cellulardecompsmoothfunctions}
u(z)=w_0(z)+M^1w_1(z)+\ldots+M^{n-1}w_{n-1}(z), \quad z \in \D,
\end{equation}
where the $M^j$ are the operators in \eqref{Moperatorfunctioncase}, and $w_j$ is the function in \eqref{olofssonsumsfunctioncase} that is given relative to some sequence of complex numbers $\{k_{m,j}\}_{m=-\infty}^{\infty}$ satisfying \eqref{limsupcoeff} for $0 \leq j \leq n-1$. 
\end{thm} 
\begin{proof}
By \eqref{almansiexpfunctions} we can write 
\begin{equation*}
u(z)=v_0(z)+(1-|z|^2)v_1(z)+\ldots+(1-|z|^2)^{n-1}v_{n-1}(z), \quad z \in \D,
\end{equation*}
for some harmonic functions $v_i \in C^{\infty}(\D)$. Introduce the functions 
\begin{equation*}
V_i(z)=M^iv_i(z), \quad z \in \D,
\end{equation*}
for $0 \leq i \leq n-1$. Express $v_i$ in the form of \eqref{harmonicsumfunctions} and write $v_{i,N}$ for its $N$:th partial sum
\begin{equation*}
v_{i,N}(z)=\sum_{m=-N}^{N}c_{i,m}\xi_m(z), \quad 0 \leq i \leq n-1,  \quad z \in \D.
\end{equation*}
Let $V_{i,N}=M^iv_{i,N}$ for $0 \leq i \leq n-1$. Then 
\begin{equation*}
V_{i,N}(z)=M^iv_{i,N}(z)=\sum_{m=-N}^{N}c_{i,m}(1-|z|^2)^i\xi_m(z),
\end{equation*}
for $0 \leq i \leq n-1$ and $z \in \D$. 

Let $l \in \N$ be such that $0 \leq l \leq n-1$ and let $m \in \Z$ be such that $-N \leq m \leq N$. By Proposition \ref{Olofssonbasis}, we can write
\begin{equation*}
(1-|z|^2)^l=\sum_{i=0}^{n-1}t_{l,i}^{|m|,n-1}\mathit{O}_{i}^{|m|,n-1}(|z|^2),
\end{equation*}
for some rational numbers $t_{l,j}^{|m|,n-1} \in \Q$ for $0 \leq j \leq n-1$, as further explained in Proposition \ref{rationalfunctionscoeff}.  Set 
\begin{equation}
\label{rationalcoeff}
t_{l,j}^{n-1}(m)=t_{l,j}^{|m|,n-1}, \quad 0 \leq j \leq n-1.
\end{equation}
Project $V_{l,N}$ onto the span of $\xi_m$ by use of the projection map $\pi_m$ from \eqref{projectionmaps}. Then 
\begin{equation*}
\pi_mV_{l,N}(z)=\sum_{i=0}^{n-1}c_{l,m}t_{l,i}^{n-1}(m)\mathit{O}_{i}^{|m|,n-1}(|z|^2)\xi_m(z), \quad z \in \D. 
\end{equation*}
We shall sum over each of these terms.

Recall that 
\begin{equation*}
\mathit{O}_{j}^{|m|,n-1}=M^j\mathit{o}_{j}^{|m|,n-1}, \quad \textnormal{on} \ \D,
\end{equation*}
for $0 \leq j \leq n-1$, where $\mathit{o}_{j}^{|m|,n-1}$ is the hypergeometric function in \eqref{hyperfunctionCdecomplemma}. Set
\begin{equation*}
W_{l,N}^j(z)=\sum_{m=-N}^{N}c_{l,m}t_{l,j}^{n-1}(m)\mathit{o}_{j}^{|m|,n-1}(|z|^2)\xi_m(z), \quad z \in \D,
\end{equation*}
for $0 \leq j \leq n-1$ and $N \geq 0$. By Proposition \ref{rationalfunctionscoeff}, we can identify each sequence $\{t_{l,j}^{n-1}(m)\}_{m \in \Z} \subset \Q$ of such numbers as the image $t_{l,j}^{n-1}(\Z)$ of a rational function $t_{l,j}^{n-1}(x) \in S_{\N}^{-1} K[x]$ for $0 \leq j \leq n-1$, noting that $t_{l,j}^{n-1}(m)=t_{l,j}^{n-1}(-m)$ for $m \in \N$. In particular then, 
\begin{equation*}
\limsup_{|m| \rightarrow \infty}|t_{l,j}^{n-1}(m)c_{l,m}|^{1/|m|} \leq 1,
\end{equation*}
for $0 \leq j \leq n-1$. By Theorem 5.1 in \cite{OK}, as presented in \eqref{pqpolrepresentations}, we then get that each of the sums $W_{l,N}^j$ converges absolutely in the space of smooth functions on $\D$ for $0 \leq j \leq n-1$. Let $W_{l}^j$ denote the resulting sum for $0 \leq j \leq n-1$.  

To conclude then, we have that each of the sums $W_{l}^j$ converges absolutely in the space of smooth functions on $\D$ for $0 \leq l \leq n-1$ and $0 \leq j \leq n-1$. Set
\begin{equation*}
w_j(z)=W_{0}^j(z)+W_{1}^j(z)+\ldots+W_{n-1}^j(z), \quad z \in \D,
\end{equation*}
for $0 \leq j \leq n-1$. Note that 
\begin{align*}
w_j(z)=\sum_{m=-\infty}^{\infty}\big(c_{0,m}t_{0,j}^{n-1}(m)+\ldots+c_{n-1,m}t_{n-1,j}^{n-1}(m)\big)\mathit{o}_{j}^{|m|,n-1}(|z|^2)\xi_m(z),
\end{align*}
for $0 \leq j \leq n-1$ and $z \in \D$. Let 
\begin{equation*}
k_{m,j}=t_{0,j}^{n-1}(m)c_{0,m}+\ldots+t_{n-1,j}^{n-1}(m)c_{n-1,m}, \quad m \in \Z, \quad 0 \leq j \leq n-1, 
\end{equation*}
and recall that $w_j$ is $(n-1-j,n-1-j)$-harmonic such that \eqref{limsupcoeff} holds. Then  
\begin{equation*}
u(z)=w_0(z)+M^1w_1(z)+\ldots+M^{n-1}w_{n-1}(z), \quad z \in \D. 
\end{equation*}
This completes the proof of the statement.
\end{proof}

The cellular decomposition for the polyharmonic functions on $\D$ is unique in the following sense.
Let $w_{1,j},w_{2,j} \in C^{\infty}(\D)$ be given as in \eqref{olofssonsumsfunctioncase} relative to some complex sequences that satisfy \eqref{limsupcoeff} for $0 \leq j \leq n-1$. Suppose that 
\begin{equation*}
w_{1,0}+Mw_{1,1}+\ldots+M^{n-1}w_{1,{n-1}}=w_{2,0}+Mw_{2,1}+\ldots+M^{n-1}w_{2,{n-1}}, \quad \textnormal{on} \ \D.
\end{equation*}
Then $w_{1,j}=w_{2,j}$ in $\D$ for $0 \leq j \leq n-1$. We will prove this shortly, following a few additional structural points in respect of the $K^{\star}$-module $K[z,\bar z]$ or $\mathcal{H}$ for short.

The $L^2(\D)$-inner product on the vector space $K[z,\bar z]$ is given by 
\begin{equation}
\label{innerproduct}
\langle p, q \rangle_K=\int_{\D} p(z,\bar z) \bar q(z,\bar z) dA(z,\bar z),
\end{equation}
where $dA$ is the normalized area measure
\begin{equation*}
dA(z,\bar z)=\frac{1}{\pi}r dr d\theta, \quad z=re^{i\theta}.
\end{equation*} 
If $(k,l)$ and $(k',l')$ are two pairs in $\N \times \N$, and we consider the basic polynomials  $p(z,\bar z)=z^k\bar z^l$  and $q(z,\bar z)=z^{k'}\bar z^{l'}$, then 
\begin{equation*}
\langle p, q \rangle_K=\int_{\D} z^k\bar z^l \bar z^{k'} z^{l'} dA(z,\bar z)=\frac{1}{\pi}\int_{0}^{1}r^{M} r dr \int_{0}^{2\pi} e^{iN\theta} d\theta,
\end{equation*}
where $M=k+l+k'+l'$ and $N=k-k'+l'-l$. Judging by the right hand side of this last equation, we see that
\begin{equation*}
\langle p, q \rangle_K = 0 \Longleftrightarrow N \neq 0 \Longleftrightarrow k-l \neq k'-l'. 
\end{equation*}
If say $k \geq l$ and $k' \geq l'$, then the condition $N=0$ translates into 
\begin{equation*}
z^k\bar z^l \bar z^{k'} z^{l'}=|z|^{2l}z^{k-l}|z|^{2l'}\bar z^{k'-l'}=|z|^{2(l+l')}z^{k-l}\bar{z}^{k'-l'}=|z|^{2(l+l')}|z|^{2(k-l)} \in K[|z|^2].
\end{equation*}
For arbitrary $k,l \in \N$ and $k',l' \in \N$, such products are suitably expressed by use of the symbol $\xi$ in \eqref{xisymbol}, in terms of which
\begin{equation*}
\xi_k(z)\xi_{-l}(z) \xi_{-k'}(z) \xi_{l'}(z)=|z|^{2\min\{k,l\}}\xi_{k-l}(z)|z|^{2\min\{k',l'\}}\xi_{-(k'-l')}(z) \in K[z,\bar z].
\end{equation*}
These examples demonstrate that we can express the condition under which two such basis elements $p=\xi_k\xi_{-l}$ and $q=\xi_{k'}\xi_{-l'}$ in $K[z,\bar z]$ are orthogonal, relative to the inner product in \eqref{innerproduct}, by saying that the polynomials $p$ and $q$ are orthogonal if and only if their product $p\bar q$ is not in $K[|z|^2] \subset K[z,\bar z]$. Equivalently, if $p$ and $q$ are not orthogonal, then their product is a function in $K[|z|^2]$ that is strictly positive on $\D \setminus \{0\}$. Integration over these basic elements extends by linearity to an arbitrary product of polynomials in $K[z,\bar z]$ if we recall  from the earlier Proposition \ref{linearindependencehomparts} that $\{\xi_m\}_{m \in \Z}$ is a generating set for the module $\mathcal{H}$ over $K^{\star}$. Indeed, let
\begin{equation*}
p(z,\bar z)=p_1\xi_{m_{p_1}}(z)+\ldots+p_{k}\xi_{m_{p_k}}(z), \quad q(z,\bar z)=q_1\xi_{m_{q_1}}(z)+\ldots+q_{l}\xi_{m_{q_l}}(z),
\end{equation*} 
for some $m_{p_j},m_{q_j} \in \Z$ and $p_1,\ldots,p_k,q_1,\ldots,q_l \in K^{\star}$. By then forming their product $p \bar q$, we get that
\begin{equation*}
p(z,\bar z)\bar q(z,\bar z)=p_1q_1\xi_{m_{p_1}}\xi_{-m_{q_1}}(z)+\ldots+p_kq_l\xi_{m_{p_k}}\xi_{-m_{q_l}}(z).
\end{equation*} 
Upon evaluating this expression with respect to the inner product in \eqref{innerproduct}, we see that the only ``surviving`` terms are those in $K[|z|^2]$. Against this background, we choose to extend the definition of the vector space inner product to the $K^{\star}$-module $\mathcal{H}$ by defining 
\begin{equation*}
\langle \cdot, \cdot \rangle_{K^{\star}}: \mathcal{H} \times \mathcal{H} \rightarrow K^{\star},
\end{equation*}
to be the module inner product on the $K^{\star}$-module $\mathcal{H}$ that is given by projecting the product $p \bar q$ onto the span of the identity element in $K[z,\bar z]$ over $K^{\star}$. Operationally, this can be done via the projection maps in \eqref{projectionmaps}, or by identifying the polynomial ring $K[z,\bar z]$ with the corresponding ring of functions and integrating over $\T$ in the form of
\begin{equation}
\label{moduleinnerproduct}
\langle p, q \rangle_{K^{\star}}=\int_{0}^{2\pi} p_r(e^{i\theta},e^{-i\theta}) \bar q_r(e^{i\theta},e^{-i \theta}) d\theta,
\end{equation} 
where $f_r(e^{i\theta})=f(re^{i\theta})$ is a function on $\T$. 
It is straightforward to check that the module inner product in \eqref{moduleinnerproduct} on $\mathcal{H}$ satisfies the following properties:
\begin{enumerate}
\item $\langle \xi_m,\xi_n \rangle_{K^{\star}}=0, \quad m \neq n, \quad m,n \in \Z$,
\item $\langle p,q \rangle_{K^{\star}}=\overline{\langle q,p \rangle}_{K^{\star}},\quad p,q \in K[z,\bar z]$, 
\item $\langle \alpha p+ \beta q, r \rangle_{K^{\star}}= \alpha \langle p,r \rangle_{K^{\star}}+\beta \langle q,r \rangle_{K^{\star}}, \quad p,q,r \in K[z,\bar z], \quad \alpha, \beta \in K^{\star}$,
\item $\langle p,p \rangle_{K^{\star}} > 0, \quad p \neq 0, \quad p \in K[z,\bar z]$. 
\end{enumerate} 
The last point (4) should here be taken in the sense of polynomial functions in $K[|z|^2]$ on $\D$ (or $\C$), whereby $p > 0$ if and only if $p(|z_0|^2)>0$ for all $z_0 \in \D \setminus \{0\}$. Thus, we may think of the $K^{\star}$-module $K[z,\bar z]$ or $\mathcal{H}$ with the module inner product $\langle \cdot, \cdot \rangle_{K^{\star}}: \mathcal{H} \times \mathcal{H} \rightarrow K^{\star}$ in a similar way to how we think of a vector space $V$ over $K$ equipped with an inner product $\langle \cdot , \cdot \rangle: V \times V \rightarrow K$, with the induced ``norm`` $||p-q||_{\mathcal{H}}=\langle p-q,p-q \rangle_{K^{\star}}$ replacing that of $||v_1-v_2||_V=\langle v_1-v_2,v_1-v_2 \rangle$.       

We shall not delve further into the reaches or limitations of the analogue between the two, but rather leave it as a source for reflection and employ it as a convenience in that which follows next. This next result shows that the cellular decomposition in Theorem \ref{cellulardecomposition} for polyharmonic functions is unique. 

\begin{lemma}
\label{uniquenesslemma}
Let $n \in \N$ be positive. Let $w_{1,j},w_{2,j} \in C^{\infty}(\D)$ be given as in \eqref{olofssonsumsfunctioncase} relative to some sequences of complex numbers $\{a_{m,j}\}_{m \in \Z}$ and $\{b_{m,j} \}_{m \in \Z}$ satisfying \eqref{limsupcoeff}, respectively, for $0 \leq j \leq n$. Let $M^j$ be the operator in \eqref{Moperatorfunctioncase} for $0 \leq j \leq n$. If
\begin{equation}
\label{differencecellulardecomp}
0=w_{1,0}(z)-w_{2,0}(z)+M^1(w_{1,1}(z)-w_{2,1}(z))+\ldots+M^{n}(w_{1,{n}}(z)-w_{2,{n}}(z)),
\end{equation} 
for $z \in \D$, then $a_{{m,j}}=b_{{m,j}}$ for all $m \in \Z$ and $0 \leq j \leq n$. Consequently, $w_{1,j}=w_{2,j}$ for $0 \leq j \leq n$. 
\end{lemma}
\begin{proof}
Let $u \in C^{\infty}(\D)$ denote the right hand side of \eqref{differencecellulardecomp} and let
\begin{equation*}
v_{j}=M^j(w_{1,j}-w_{2,j}), \quad 0 \leq j \leq n. 
\end{equation*}
Let $m \in \Z$. By property (3) for the module inner product above, 
\begin{align*}
0&=\langle u,\xi_m \rangle_{K^{\star}}=\langle v_{0}+v_{1}+\ldots+v_{n},\xi_m \rangle_{K^{\star}}=\langle v_{0},\xi_m \rangle_{K^{\star}}+\ldots+\langle v_{n},\xi_m \rangle_{K^{\star}},
\end{align*}
where $\langle \cdot, \cdot \rangle_{K^{\star}}$ is to be understood in the sense of \eqref{moduleinnerproduct}. With the use of property (1) for the module inner product in \eqref{moduleinnerproduct}, we then see that
\begin{equation*}
0=(a_{{m,0}}-b_{{m,0}})\mathit{o}_{0}^{|m|,n}+(a_{{m,1}}-b_{{m,1}})M\mathit{o}_{1}^{|m|,n}+\ldots+(a_{{m,n}}-b_{{m,n}})M^n\mathit{o}_{n}^{|m|,n},
\end{equation*}
on $\D$. Recall that 
\begin{equation*}
\mathit{O}_{j}^{|m|,n}(|z|^2)=(1-|z|^2)^j\mathit{o}_{j}^{|m|,n}(|z|^2), \quad 0 \leq j \leq n, \quad z \in \D. 
\end{equation*}
Expressed in such terms, we get that 
\begin{equation*}
0=(a_{{m,0}}-b_{{m,0}})\mathit{O}_{0}^{|m|,n}+(a_{{m,1}}-b_{{m,1}})\mathit{O}_{1}^{|m|,n}+\ldots+(a_{{m,n}}-b_{{m,n}})\mathit{O}_{n}^{|m|,n},
\end{equation*}
on $\D$. Since the functions $\mathit{O}_{0}^{|m|,n},\ldots,\mathit{O}_{n}^{|m|,n}$ are linearly independent on $\D$ by Lemma \ref{Cdecomplemma}, we may now conclude with
\begin{equation*}
a_{{m,j}}-b_{{m,j}}=0,  
\end{equation*} 
for $0 \leq j \leq n$.
\end{proof}
It follows from the last lemma that the cellular decomposition is unique in the sense that was described following Theorem \ref{cellulardecomposition}. 
\begin{cor}
Let $n \in \N$ be positive. Let $u \in C^{2n}(\D)$ be a polyharmonic function of order $n$ on $\D$. Then the cellular decomposition for $u$ in \eqref{cellulardecompsmoothfunctions} is unique. 
\end{cor}
\begin{proof}
The uniqueness statement is a consequence of Lemma \ref{uniquenesslemma}.
\end{proof}

There is a ``meta-narrative`` in respect of Theorem \ref{cellulardecomposition} that is worth mentioning before we end. For as we saw in the last section and Proposition \ref{polyharmonicorder}, the poly\-harmonic functions of order $n \in \N$ grow out of the basic elements $z^k \bar z^l$ where either $k < n$ or $l < n$. Meanwhile, the generalized harmonic operators $L_{\gamma_1,\gamma_2} \in \mathfrak{H}_2$ can be written as
\begin{equation*}
L_{\gamma_1,\gamma_2}=\partial \bar \partial-(\bar z \bar \partial-\gamma_1)(z\partial-\gamma_2).
\end{equation*}
If we now evaluate this last operator with respect to such products of monomials, there results an expression
\begin{equation*}
L_{\gamma_1,\gamma_2} z^k \bar z^l=\partial \bar \partial z^k \bar z^l-(l-\gamma_1)(k-\gamma_2)z^k \bar z^l.
\end{equation*}
When $\gamma_1=l \in \N$ or $\gamma_2=k \in \N$, one or the other of the two expressions in the parentheses is zero, and $L_{\gamma_1,\gamma_2} z^k \bar z^l=\partial \bar \partial z^k \bar z^l \in \mathcal{H}^{n-1}$. In particular then, we see that 
\begin{equation*}
L_{n-1,n-1}\mathcal{H}^n \subset \mathcal{H}^{n-1}. 
\end{equation*}
This last observation was first made by A.~Borichev and H.~Hedenmalm in \cite{BH}, and was deduced from their factorization of the polyharmonic operators in terms of such generalized harmonic operators, as described by the authors in their Corollary 6.3. This description can also be compared to the discussion at the beginning of section five. 

\section*{Acknowledgements}

The author would like to thank Anders Olofsson for the many reflective dialogues on the topic of the present paper. The author is grateful to Jens Wittsten for his detailed and valuable suggestions in relation to this text. 

\bigskip

Research was partially supported by the Swedish Research Council grant number 2019-04878.


\end{document}